\documentclass[10pt]{amsart}
\usepackage{amsmath,amsthm,amsfonts,amssymb,amscd,flafter,pinlabel}
\usepackage{mathtools}
\usepackage[mathscr]{eucal}
\usepackage{graphics}
 \usepackage[all]{xy}
\usepackage{caption, subcaption}
\usepackage[usenames,dvipsnames]{color}
\usepackage{graphicx}

\usepackage{pgf}
\usepackage{tikz}
\usetikzlibrary{arrows,automata}
\usepackage[latin1]{inputenc}

\usepackage[colorlinks=true, urlcolor=NavyBlue, linkcolor=NavyBlue, citecolor=NavyBlue, pdftitle={On Conway mutation and link homology}, pdfauthor={Peter Lambert-Cole}, pdfsubject={Knot Floer Homology, Khovanov homology, mutation}, pdfkeywords={Heegaard Floer homology, 57M27, 57R58}]{hyperref}

\newtheorem{theorem}{Theorem}[section]
\newtheorem{corollary}[theorem]{Corollary}
\newtheorem{conjecture}[theorem]{Conjecture}
\newtheorem{question}[theorem]{Question}
\newtheorem{proposition}[theorem]{Proposition}

\newtheorem{lemma}[theorem]{Lemma}

\theoremstyle{definition}
\newtheorem{definition}[theorem]{Definition}

\theoremstyle{definition}
\newtheorem{remark}[theorem]{Remark}

\newcommand{\Q}{\ensuremath{\mathbb{Q}}}
\newcommand\HFh{{\rm {\widehat{HF}}}}
\newcommand\HFK{{\rm {HFK}}}
\newcommand\HFKh{{\rm {\widehat{HFK}}}}
\newcommand\HFKt{{\rm {\widetilde{HFK}}}}

\newcommand\CFK{{\rm {CFK}}}

\newcommand\CFKm{{\rm {CFK^{-}}}}
\newcommand\CFKt{\widetilde{\mathrm{CFK}}}

\newcommand\alphas{\mbox{\boldmath$\alpha$}}
\newcommand\betas{\mbox{\boldmath$\beta$}}

\newcommand\ws{\mathbf w}
\newcommand\zs{\mathbf z}
\newcommand\FF{\mathbb F}

\newcommand\TT{\mathbb{T}}

\newcommand\SH{\mathcal{H}}
\newcommand\rk{\rm{rk }\,}

\def\x{\mathbf{x}}

\def\y{\mathbf{y}}
\def\z{\mathbf{z}}
\def\p{\mathbf{p}}
\def\w{\mathbf{w}}

\newcommand{\cM}{\mathcal{M}}

\newcommand{\RR}{\mathbb{R}}
\newcommand{\QQ}{\mathbb{Q}}
\newcommand{\del}{\partial}

\newcommand{\ZZ}{\mathbb{Z}}

\newcommand{\cL}{\mathcal{L}}

\newcommand{\cU}{\mathcal{U}}

\newcommand{\cT}{\mathcal{T}}
\newcommand{\cP}{\mathcal{P}}
\newcommand{\cA}{\mathcal{A}}
\newcommand{\cC}{\mathcal{C}}

\newcommand{\Kh}{{\rm {Kh }}}
\newcommand{\CKh}{{\rm {CKh }}}
\newcommand{\Khr}{\widetilde{\Kh }}

\newcommand{\XX}{\mathbb{X}}
\newcommand{\OO}{\mathbb{O}}
\newcommand{\GG}{\mathbb{G}}
\newcommand{\St}{\mathbf{S}}

\newcommand{\Cone}{{\rm {Cone}}}
\newcommand{\Ker}{{\rm {Ker}}}
\newcommand{\Image}{{\rm {Im}}}
\newcommand{\Coker}{{\rm {Coker}}}
\newcommand{\Link}{\mathbf{Link}}
\newcommand{\Diag}{\mathbf{Diag}}
\newcommand{\Spect}{\mathbf{Spect}}
\newcommand{\Vect}{\mathbf{Vect}}

\newcommand{\Rect}{{\rm {Rect}}}
\newcommand{\Pent}{{\rm {Pent}}}

\newcommand{\Tri}{{\rm {Tri}}}

\newcommand{\GC}{{\rm {GC}}}
\newcommand{\GH}{{\rm {GH}}}
\newcommand{\GCt}{\widetilde{\GC}}

\newcommand{\GHt}{\widetilde{\GH}}
\newcommand{\GHh}{\widehat{\GH}}

\newcommand{\leftbracket}{[}

\begin{document}

\title[{On Conway mutation and link homology}]{On Conway mutation and link homology}

\author[P. Lambert-Cole]{Peter Lambert-Cole}
\address{Department of Mathematics \\ Indiana University}
\email{pblamber@indiana.edu}
\urladdr{\href{https://www.pages.iu.edu/~pblamber/}{https://www.pages.iu.edu/\~{}pblamber}}

\maketitle


\begin{abstract}

We give a new, elementary proof that Khovanov homology with $\ZZ/2\ZZ$--coefficients is invariant under Conway mutation.  This proof also gives a strategy to prove Baldwin and Levine's conjecture that $\delta$--graded knot Floer homology is mutation--invariant.  Using the Clifford module structure on $\HFKt$ induced by basepoint maps, we carry out this strategy for mutations on a large class of tangles.  Let $L'$ be a link obtained from $L$ by mutating the tangle $T$.  Suppose some rational closure of $T$ corresponding to the mutation is the unlink on any number of components.  Then $L$ and $L'$ have isomorphic $\delta$--graded $\HFKh$ groups over $\ZZ/2\ZZ$ as well as isomorphic Khovanov homology over $\QQ$.  We apply these results to establish mutation--invariance for the infinite families of Kinoshita-Terasaka and Conway knots.  Finally, we give sufficient conditions for a general Khovanov-Floer theory to be mutation--invariant.

\end{abstract}

\section{Introduction}

A {\it Conway sphere} for a link $L \subset S^3$ is a smoothly embedded 2-sphere that intersects the link in 4 points.  This separates the link into a pair of tangles $(B^3,T_1)$ and $(B^3,T_2)$.  The center of the mapping class group of $S^2$ with 4 marked points has four elements: the identity and 3 involutions.  After identifying $S^2$ with the unit sphere in $\RR^3$ and the four marked points on the $xz$--plane, we can identify the 3 involutions with rotations by $\pi$ around the 3 coordinate axes.  A {\it mutation} of $L$ is a link obtained by changing the gluing map of $(B^3,T_1)$ and $(B^3,T_2)$ by such an involution $\tau$.

Let $\HFKh(L)$ denote the $\ZZ \oplus \ZZ$--graded knot Floer homology groups associated to the link $L$ with coefficients in $\FF_2$.  The $\delta$--graded knot Floer groups are obtained by collapsing along the diagonals $m - a = \delta$:
\[ \HFKh_{\delta}(L) \coloneqq \bigoplus_{m - a = \delta} \HFKh_m(L,a)\]

The bigraded knot Floer homology groups can distinguish mutant knots, such as the Kinoshita-Terasaka and Conway knots \cite{OS-mutation,BG-computations}.  However, explicit computations showed that for 11- and 12-crossing knots, mutation preserves the $\delta$--graded invariant \cite{BG-computations}.  This data, along with a combinatorial model for $\HFKh_{\delta}$, led Baldwin and Levine to conjecture that this phenomenon is true in general.  

\begin{conjecture}[Baldwin-Levine \cite{BL-spanning}]
\label{conj:mutation-invariance}
Let $L$ and $L'$ be a mutant pair of links.  Then there is an isomorphism
\[\HFKh_{\delta}(L) \cong \HFKh_{\delta}(L')\]
\end{conjecture}

In this paper, we investigate this conjecture and prove it for mutations on a class of tangles.  Let $(B^3,T_0)$ denote the tangle consisting of 2 boundary-parallel arcs.  A link $L$ is a {\it rational closure} of a tangle $(B^3,T)$ if it can be decomposed in the form
\[(S^3, L) = (B^3_1,T) \cup_{\phi} (B^3_2,T_0)\]
for some homeomorphism $\phi: \del B^3_1 \rightarrow \del B^3_2$.  Rational closures are not unique: the set of rational closures of $(B^3,T_0)$ itself is the set of 2-bridge links.  Let $\cC(T)$ denote the set of rational closures of $T$.  A mutation of $T$ by the involution $\tau$ determines a subset $\cC_{\tau}(T)$ as follows.  A rational closure of $T$ is in $\cC_{\tau}(T)$ if the arcs of $T_0$ connect points of $\del T$ exchanged by the involution $\tau$.  We refer to $\cC_{\tau}(T)$ as the set of {\it rational closures corresponding to the mutation $\tau$}.

In this paper, we show that if the set $\cC_{\tau}(T)$ contains the unlink on any number of components, then many link homology theories are preserved by the mutation $\tau$.  For example, Kinoshita and Terasaka defined an infinite family of knots $KT_{r,n}$ for $r,n \in \ZZ$ with trivial Alexander polynomial, where the Kinoshita-Terasaka  knot $11n42$ is $KT_{2,1}$ \cite{KT}.  Analogously, there is an infinite family of Conway mutants $C_{r,n}$ extending $C_{2,1} = 11n34$ that are obtained from the Kinoshita-Terasaka family by mutation.  We can choose a diagram for $KT_{r,n}$ so that the numerator closure of the mutated tangle is the unknot (Figure \ref{fig:KT}).  As a second example, if $T$ is the tangle sum of two rational tangles $\left[\frac{p}{q}\right]$ and $\left[\frac{r}{s}\right]$, then the numerator closure of $T$ is the unknot or 2-component unlink if $ps + qr \in \{-1,0,1\}$ \cite{Kauffman-Lambro}.

\subsection{Khovanov Homology}

The starting point is a new proof of the following theorem.

\begin{theorem}[Bloom \cite{Bloom-Odd-Khovanov}, Wehrli \cite{Wehrli-Kh-mutation}]
\label{thrm:Kh-mutation-invariance}
Let $L,L'$ be mutant links and let $\Kh$ denote Khovanov homology with $\ZZ/2\ZZ$-coefficients.  There is an isomorphism
\begin{align*}
\Kh(L) &\cong \Kh(L')
\end{align*}
\end{theorem}

The tools involved are more elementary than the proofs in \cite{Bloom-Odd-Khovanov,Wehrli-Kh-mutation}.  In particular, we need only the following 3 facts\footnote{Facts (2) and (3) themselves are consequences of the fact that over $\ZZ/2\ZZ$, reduced Khovanov homology does not depend on the component containing the basepoint.}:
\begin{enumerate}
\item the unoriented skein exact triangle
\[ \rightarrow \Kh(L) \rightarrow \Kh(L_0) \rightarrow \Kh(L_1) \rightarrow \Kh(L) \rightarrow\]
\item over $\ZZ/2 \ZZ$ and for any two (unoriented) links $L_1,L_2$, the group $\Kh(L_1 \# L_2)$ is independent, up to isomorphism, of the choice of connected sum of $L_1$ and $L_2$, and
\item over $\ZZ/2 \ZZ$ and for any two (unoriented) links $L_1,L_2$, the elementary merge cobordism $$\mu: \Kh(L_1 \cup L_2) \rightarrow \Kh(L_1 \# L_2)$$ is surjective for any choice of connected sum.
\end{enumerate}

Our strategy is to find diagrams for a mutant pair $L,L'$ in {\it standard form} (Figure \ref{fig:mutation-standard}).  We then apply the unoriented skein exact sequence to relate $\Kh(L)$ and $\Kh(L')$. 

\begin{figure}[htpb!]
\centering
\labellist
	\small\hair 2pt
	\pinlabel $T_1$ at 50 100
	\pinlabel $T_2$ at 242 100
	\pinlabel $T_1$ at 400 100
	\pinlabel $T_2$ at 592 100
\endlabellist
\includegraphics[width=.45\textwidth]{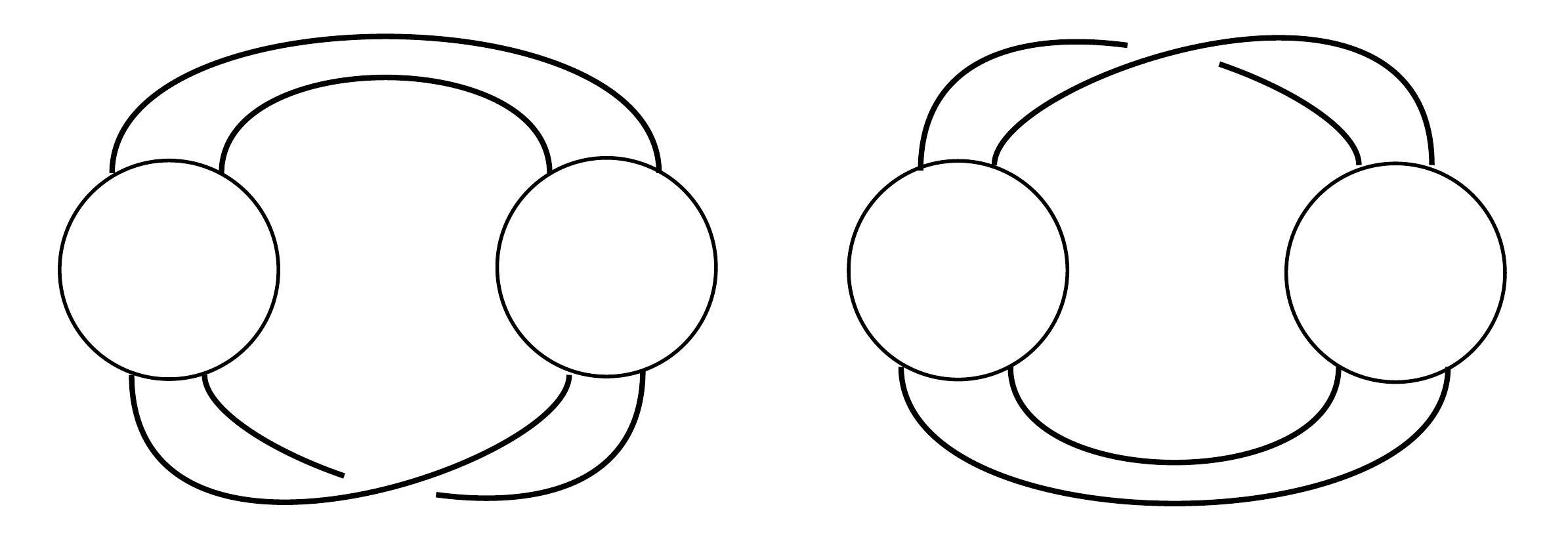}
\caption{A mutant pair in standard form}
\label{fig:mutation-standard}
\end{figure}
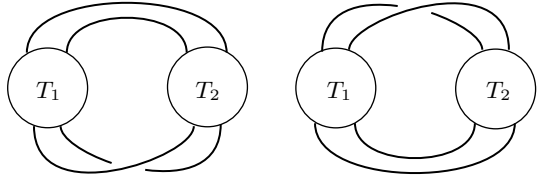

Over a different coefficient ring or with respect to a different link homology theory, Facts (2) and (3) may or may not hold.  In particular, Facts (2) and (3) are not true in general for Khovanov homology over $\ZZ$ or $\QQ$.  Moreover, it is exactly the failure of (2) that leads to Wehrli's examples of links distinguished by $\Kh$ with $\QQ$ coefficients \cite{Wehrli-KH-examples}.  With extra conditions, we can adapt this proof of mutation invariance to other coefficient rings or link homology theories.

\begin{theorem}
\label{thrm:Kh-QQ-mutation-invariance}
Let $L,L'$ be links such that $L'$ is obtained from $L$ by mutating the tangle $T$ by the involution $\tau$.  Let $\cC_{\tau}(T)$ denote the set of rational closures of $T$ corresponding to the mutation.  If $\cC_{\tau}(T)$ contains the unlink on any number of components, then for any field $\FF$ there is a bigraded isomorphism
\[\Kh_{\FF}(L) \cong \Kh_{\FF}(L')\]
In particular
\[\Kh_{\QQ}(L) \cong \Kh_{\QQ}(L')\]
\end{theorem}

\begin{remark}
Theorem \ref{thrm:Kh-QQ-mutation-invariance} applies to {\it all} mutations, not simply component--preserving mutations.  However, it only applies to unreduced Khovanov homology.  If $L,L'$ contain a basepoint $p$ in $T$, then a mutation satisfying the hypotheses can swap the component containing $p$.
\end{remark}

A corollary to Theorem \ref{thrm:Kh-QQ-mutation-invariance} is that all mutant pairs in the Kinoshita-Terasaka and Conway families have isomorphic Khovanov homology.

\begin{theorem}
\label{thrm:KT-mutation-Khovanov}
For all $r,n \in \ZZ$ and any field $\FF$, there is a bigraded isomorphism
\[\Kh_{\FF}( KT_{r,n}) \cong \Kh_{\FF}( C_{r,n})\]
\end{theorem}

\subsection{Knot Floer homology}

A variant of knot Floer homology, the tilde group $\HFKt(L,\p)$ of a pointed link, satisfies an unoriented skein exact triangle \cite{Manolescu-skein}.  Moreover, the invariant $\HFKh$ satisfies a Kunneth formula for any choice of connected sum \cite{OS-HFK,OS-HFL}.  However, the elementary merge map
\[\mu: \HFKt(L_1 \cup L_2, \p) \rightarrow \HFKt(L_1 \# L_2,\p)\]
is never surjective.  In fact, the rank of $\mu$ is exactly $\frac{1}{2} \rk \HFKt(L_1 \cup L_2,\p) = \frac{1}{2} \rk \HFKt(L_1 \# L_2,\p)$.  Nonetheless, the proof of Theorem \ref{thrm:Kh-mutation-invariance} provides a strategy to prove Conjecture \ref{conj:mutation-invariance}.  The main theorem of this paper is to carry out that strategy for mutations on a large class of tangles.  This provides the first step in the proof of Conjecture \ref{conj:mutation-invariance}.

\begin{theorem}
\label{thrm:HFK-mutation-invariance}
Let $L,L'$ be links such that $L'$ is obtained from $L$ by mutating the tangle $T$ by the involution $\tau$.  Let $\cC_{\tau}(T)$ denote the set of rational closures of $T$ corresponding to the mutation.  If $\cC_{\tau}(T)$ contains the unlink on any number of components, then there is an isomorphism
\[\HFKh_{\delta}(L') \cong \HFKh_{\delta}(L)\]
\end{theorem}

A key technical tool in the proof of Theorem \ref{thrm:HFK-mutation-invariance} is the extra algebraic structure on unreduced knot Floer homology $\HFKt(L,\p)$ of a pointed link $(L,\p)$.   Each basepoint in a multi-pointed Heegaard diagram for $(L,\p)$ determines a differential on $\HFKt(L,\p)$.  These basepoint actions have previously been applied in \cite{BL-spanning,BVV,Sarkar-basepoints,BLS,Zemke-basepoints}.  The combined actions, subject to anticommutation relations described in Subsection \ref{sub:basepoints}, make $\HFKt(L,\p)$ a Clifford module over a Clifford algebra $\Omega_{\nu}$.  This Clifford module structure is an algebraic quantization of the exterior algebra structure used in \cite{BL-spanning,BLS}.

Importantly, the maps in the unoriented skein exact triangle are mostly but not fully compatible with the basepoint actions.  In particular, the three homology groups in the skein triangle are modules over slightly different Clifford algebras.  In Subsection \ref{sub:basepoint-equivariance}, we quantify the skein maps' failure to be fully $\Omega_{\nu}$-linear.  Nonetheless, this weaker equivariance condition is sufficient to prove the theorem.

We remark that the following statement is an easy corollary of Theorem \ref{thrm:HFK-mutation-invariance}.

\begin{corollary}
Let $L,L'$ be links satisfying the hypotheses of Theorem \ref{thrm:HFK-mutation-invariance} and suppose that $\HFKh(L)$ is thin.  Then $\HFKh(L')$ is also thin.  Moreover, if $L,L'$ are knots then
\begin{enumerate}
\item $g(L') = g(L)$,
\item $L'$ is fibered if and only if $L$ is fibered, and
\item $\tau(L) = \tau(L')$
\end{enumerate}
\end{corollary}

We can apply Theorem \ref{thrm:HFK-mutation-invariance} in a few ways.  First, Ozsv{\'a}th and Szab{\'o} computed the top--degree $\HFKh$ groups for the Kinoshita-Terasaka and Conway families and showed that $KT_{r,n}$ and $C_{r,n}$ are distinguished by their bigraded knot Floer groups \cite{OS-mutation}.  This extends an earlier result of Gabai calculating their genera \cite{Gabai-genera}.  However, using Theorem \ref{thrm:HFK-mutation-invariance}, we can deduce that $KT_{r,n}$ and $C_{r,n}$ have isomorphic $\delta$--graded groups.

\begin{theorem}
\label{thrm:KT-mutation}
For all $r,n \in \ZZ$, there is a graded isomorphism
\[\HFKh_{\delta}(KT_{r,n}) \cong \HFKh_{\delta}(C_{r,n})\]
\end{theorem}

Secondly, De Wit and Links and, independently, Stoimenow give enumerations of 11- and 12-crossing mutant cliques (6 alternating and 10 nonalternating 11-crossing pairs; 27 pairs and 2 triples of alternating 12-crossing knots; 43 pairs and 3 triples of nonalternating 12-crossing knots) \cite{DeWit-Links,Stoimenow}.  The $\delta$--graded knot Floer groups of an alternating link is determined by the determinant and signature \cite{OS-alternating}.  Since these are mutation--invariant, Conjecture \ref{conj:mutation-invariance} holds for alternating links.  Moreover,  all nonalternating cliques in the De Wit-Links and Stoimenow enumerations admit a mutation in a minimal diagram on one of the tangles in Figure \ref{fig:3-tangles}.  The numerator closures of these 3 tangles are unlinked.  Applying Theorem \ref{thrm:HFK-mutation-invariance}, we recover the computational results of \cite{BG-computations}.

\begin{theorem}
\label{thrm:low-crossing}
Let $K,K'$ be mutant knots with crossing number at most 12.  Then
\[\HFKh_{\delta}(K) \cong \HFKh_{\delta}(K')\]
for all $\delta \in \ZZ$.
\end{theorem}

\subsection{Khovanov-Floer theories}
\label{sub:intro-KF}

Interestingly, the geometric arguments in the proof of Theorem \ref{thrm:Kh-mutation-invariance} apply more generally to other link homology theories.

Baldwin, Hedden and Lobb introduced the notion of a {\it Khovanov-Floer theory} \cite{BHL-KF}.  See Section \ref{sec:KF} for a definition.  This framework encompasses several link homology theories --- Heegaard Floer homology of the double-branched cover $\Sigma(L)$ \cite{OS-2cover}, singular instanton homology \cite{KM-Kh}, Szab{\'o}'s cube of resolutions \cite{Szabo-GSS}, Bar-Natan's construction of Lee homology over $\ZZ/2\ZZ$ \cite{BN-Kh} --- that admit similar spectral sequences.  In particular, the homology groups can be computed from filtered complexes and the $E^2$--pages of the corresponding spectral sequences are isomorphic to Khovanov homology.  Moreover, many of these theories appear insensitive to Conway mutation.  

It is possible, but unknown, that all Khovanov-Floer theories must also satisfy extra elementary properties of Khovanov homology with $\ZZ/2\ZZ$--coefficients.  In order to give sufficient conditions for a Khovanov-Floer theory to be mutation--invariant, we introduce the notion of an extended Khovanov-Floer theory.  

\begin{definition}
An {\it extended Khovanov-Floer theory} is a pair $\cA,\cA^r$ consisting of an unreduced and reduced Khovanov-Floer theories, respectively, satisfying the following 3 extra axioms:

\begin{enumerate}
\item $\cA(L) = \cA^r(L \cup U,p)$ for a basepoint $p$ on the unknot component $U$,
\item $\cA$ and $\cA^r$ satisfy unoriented skein exact triangles, and
\item up to isomorphism, $\cA^r(L,p)$ is independent of the component containing $p$.
\end{enumerate}
\end{definition}

The proof of Theorem \ref{thrm:Kh-mutation-invariance} in Section \ref{sec:Kh} easily adapts to prove the following theorem.  We state it without reference to grading, although we expect it can be extended to a graded statement by inspecting the relevant exact triangle.

\begin{theorem}
\label{thrm:KF-mutation-invariance}
Let $\cA,\cA^r$ be an extended Khovanov-Floer theory.  Then $\cA$ and $\cA^r$ are invariant under Conway mutation.
\end{theorem}

For example, let $\Sigma(L)$ denote the double branched cover of $L \subset S^3$ and let $\HFh$ denote the Heegaard Floer homology with $\ZZ/2\ZZ$--coefficients.  Setting $\cA^r(L,p) = \HFh(-\Sigma(L))$ for any choice of basepoint $p$ and $\cA(L) = \HFh(-\Sigma(L) \# S^1 \times S^2)$ yields an extended Khovanov-Floer theory.  The mutation--invariance of $\Sigma(L)$, and thus $\HFh(-\Sigma(L))$, is a well--known fact \cite{Viro}. 

Another potential extended Khovanov-Floer theory is Szab{\'o}'s geometric spectral sequence \cite{Szabo-GSS}.  For a link $L$, each link diagram $\mathcal{D}$ and {\it decoration} $\mathbf{t}$ determines a filtered chain complex $\widehat{C}(\mathcal{D},\mathbf{t})$.  The pages of the corresponding spectral sequence are independent of $\mathcal{D}$ and $\mathbf{t}$ and so are invariants of $L$.  The complex $\widehat{C}(\mathcal{D},\mathbf{t})$ is constructed via a cube of resolutions, so its homology satisfies an unoriented skein triangle.  Moreover, there are naturally reduced complexes $\overline{C}(\mathcal{D},\mathbf{t}), \widetilde{C}(\mathcal{D},\mathbf{t})$ determined by a basepoint on $L$.  Extensive computational work by Seed is consistent with the conjecture that $H_*(\widehat{C}(\mathcal{D},\mathbf{t})) = \HFh(-\Sigma(L))$, that the reduced homology is independent of the basepoint, and that each page of the spectral sequence is mutation--invariant \cite{Seed-comp}.  Thus we conjecture the following:

\begin{conjecture}
Szab{\'o}'s geometric spectral sequence is an extended Khovanov-Floer theory.
\end{conjecture}

Finally, the singular instanton homology groups $I^{\#},I^{\natural}$ defined by Kronheimer and Mrowka are known to satisfy the Khovanov-Floer axioms \cite{BHL-KF}.  In addition, they satisfy an unoriented skein exact triangle.

\begin{question}
\label{question:SIH}
Are the singular instanton homology groups $I^{\#}(L),I^{\natural}(L)$ an extended Khovanov-Floer theory?
\end{question}

A positive answer to Question \ref{question:SIH} would imply, via Theorem \ref{thrm:KF-mutation-invariance}, that the singular instanton homology groups $I^{\#},I^{\natural}$ are mutation--invariant over $\ZZ/2\ZZ$.  Regardless of the answer, however, we can prove an analogous statement to Theorems \ref{thrm:Kh-QQ-mutation-invariance} and \ref{thrm:KT-mutation-Khovanov}.

\begin{theorem}
\label{thrm:SIH-mutation}
Let $L,L'$ be links such that $L'$ is obtained from $L$ by mutating the tangle $T$ by the involution $\tau$.  Let $\cC_{\tau}(T)$ denote the set of rational closures of $T$ corresponding to the mutation.  If $\cC_{\tau}(T)$ contains the unlink on any number of components, then for any field $\FF$ there is an isomorphism
\[I^{\#}_{\FF}(L) \cong I^{\#}_{\FF}(L')\]
Consequently, for all $r,n \in \ZZ$, there is an isomorphism
\[I^{\#}_{\FF}(KT_{r,n}) \cong I^{\#}_{\FF}(C_{r,n}) \]
\end{theorem}

\subsection{Acknowledgements}

I would like to thank Matt Hogancamp for many useful discussions on homological algebra.  In addition, several people have helped with apt suggestions, technical details and their general interest, including John Baldwin, Matt Hedden, Adam Levine, Tye Lidman and Zoltan Szab{\'o}.

\section{Khovanov mutation invariance}
\label{sec:Kh}

In this section, we give a new proof that Khovanov homology with $\ZZ/2\ZZ$ coefficients is invariant under Conway mutation.  Similar geometric arguments will be applied in successive sections to establish mutation--invariance results for knot Floer homology and other Khovanov-Floer theories.

Bloom proved that odd Khovanov homology is invariant under mutation, which implies that (even) Khovanov homology with $\ZZ/2\ZZ$ coefficients is also invariant \cite{Bloom-Odd-Khovanov}.  In addition, Wehrli showed that Bar-Natan's Khovanov bracket over $\ZZ/2\ZZ$ is invariant under component--preserving mutation \cite{Wehrli-Kh-mutation}.  Conversely, Wehrli has also observed that Khovanov homology is not mutation--invariant with $\QQ$ coefficients.  The links $T(2,3) \cup T(2,3)$ and $T(2,3) \# T(2,3) \cup U$ are mutants but are distinguished by their Khovanov groups over $\QQ$ \cite{Wehrli-KH-examples}.

\subsection{Khovanov homology}

Khovanov homology is an oriented link invariant obtained by applying a (1+1)-dimensional TQFT to the cube of resolutions of a link diagram.  In this first subsection, we will sketch a definition of Khovanov homology, state some well--known properties and prove some elementary properties.

Let $D$ be a planar diagram of an oriented link $L$ with $n$ crossings and fix an enumeration of the crossings.  Each given crossing can be resolved in two ways, the 0--resolution and 1--resolution.  The {\it cube of resolutions} of $D$ is the collection of $2^n$ planar diagrams obtained by resolving all the crossings of $D$ in all possible ways.  In particular, for each vertex $I \in \{0,1\}^n$, there is a diagram $D_I$ of an unlink obtained by resolving the crossings of $D$ according to the vector $I$.  Let $l_I$ denote the number of link components of $D_I$.  The edges of the cube are given by ordered pairs $I,J$ of vertices such that $I_k \leq J_k$ for $k=1,\dots,n$ and $|J| - |I| = 1$.  Geometrically, each edge corresponds to replacing a single 0--resolution by a 1--resolution and either {\it merges} two components into one or {\it splits} one component into two.

The algebra $A \coloneqq \ZZ [x]/x^2$ is a Frobenius algebra with multiplication $\mu: A \otimes A \rightarrow A$ defined by
\begin{align*}
\mu(1 \otimes 1) &\coloneqq 1 & \mu(1 \otimes x) &\coloneqq x &
\mu(x \otimes 1) &\coloneqq x & \mu(x_1\otimes x) &\coloneqq 0
\end{align*}
and comultiplication $\sigma: A \rightarrow A \otimes A$ defined by
\begin{align*}
\sigma(1) &\coloneqq x \otimes 1 + 1 \otimes x & \sigma(x) &\coloneqq x \otimes x
\end{align*}
To each vertex $I$ of the cube assign the chain group $\CKh(D_I) \coloneqq A^{\otimes l_I}$, with one copy of $A$ for each component of $D_I$.  The chain group of the Khovanov complex is the direct sum of the chain groups for each vertex
\[\CKh(D) \coloneqq \bigoplus_{I \in \{0,1\}^n} \CKh(D_I)\]
The differential $d_{\Kh}$ is the sum of maps associated to each edge.  For each edge $I,J$ of the cube, there is a component $d_{I,J}$ of the differential determined up to sign by the Frobenius algebra.  If the edge is a merge map, then $d_{I,J}$ is defined by applying the multiplication map $\mu$ to the $A$-factors of $\CKh(D_I)$ corresponding to the merged components and extending this by the identity to the remaining components.  If the edge is a split map, then $d_{I,J}$ is defined similarly using the comultiplication $\sigma$ instead.  {\it Khovanov homology} $\Kh(D)$ is the homology of the complex $(\CKh(D),d_{\Kh})$.

The complex $(\CKh(D),d_{\Kh})$ possess two gradings, the {\it quantum} and {\it homological} grading, and the differential preserves the quantum grading and increases the homological grading by one.  Thus the homology $\Kh(D)$ splits into the direct sum of bigraded modules $\Kh^{i,j}(D)$ where $i$ denotes the homological grading and $j$ the quantum grading.

Let $p$ be a fixed basepoint in the plane contained in the diagram $D$.  The basepoint $p$ determines a chain map $X_p: \CKh(D) \rightarrow \CKh(D)$ that squares to 0.  The kernel of $X_p$ is a subcomplex and {\it reduced Khovanov homology} is its homology, with a shift in the quantum grading:
\[ \Khr(D,p) = H_*(\Sigma^{0,1} \ker X_p)\]
It is an invariant of $L$ up to isotopies supported away from $p$.  While over $\ZZ$ the reduced homology depends on the component containing the basepoint, this is not true over $\ZZ/2\ZZ$.
\begin{proposition}[\cite{Shumakovitch,OS-2cover}]
\label{prop:basepoint-independence}
Over $\ZZ/2\ZZ$, there is a bigraded isomorphism
\[\Kh (L) \cong \Khr(L,p) \otimes \Kh (U)\]
for every link $L$ and any basepoint $p$.  In particular, $\Khr(L)$ is well--defined independent of $p$.
\end{proposition}
Consequently, when discussing reduced Khovanov homology over $\ZZ/2\ZZ$ we will suppress any mention of the basepoint $p$.

Khovanov homology satisfies Kunneth--type formulas for disjoint unions and connected sums.  The following properties are well--known and the proofs are easy deductions from the definition of Khovanov homology and Proposition \ref{prop:basepoint-independence}.

\begin{lemma}
\label{lemma:Kh-Kunneth}
Let $L_1,L_2$ be arbitrary links.  Over $\ZZ/2\ZZ$ there are isomorphisms
\begin{align*}
\Kh(L_1 \cup L_2) & \cong \Kh(L_1) \otimes \Kh(L_2) \\
\Khr(L_1 \# L_2) & \cong \Khr(L_1) \otimes \Khr(L_2) \\
\Kh(L_1 \# L_2) \otimes \Kh(U) &\cong \Kh(L_1) \otimes \Kh(L_2) 
\end{align*}
for any choice of connected sum.
\end{lemma}

Khovanov homology satisfies an {\it unoriented skein exact triangle}.  Fix a crossing of $D$ and let $D_0,D_1$ denote the 0-- and 1--resolutions of $D$ at this crossing.  Since Khovanov homology is computed from a cube of resolutions, the complex $\CKh(D)$ is, up to a grading shift, the mapping cone of a chain map
\[f: \CKh(D_0) \rightarrow \CKh(D_1)\]
Consequently there is an exact triangle
\[\xymatrix{
\Kh(D_0) \ar[rr]^{f^*} && \Kh(D_1) \ar[ld] \\
& \Kh(D) \ar[lu]^{[-1]} & }\]
The bigrading shift of $f$ depends on whether the resolved crossing is positive or negative.  Let $n_-$ be the number of negative crossings in $D$ and let $n_-^0,n_-^1$ be the number of negative crossings in $D_0$ and $D_1$, respectively.  Then the Khovanov differential determines chain maps 
\begin{align*}
f_+: &\CKh^{i,j}(D_0) \rightarrow \CKh^{i-c,j-1-3c}(D_1) & f_-:& \CKh^{i-d-1,j-2-3d}(D_0) \rightarrow \CKh^{i,j}(D_1)
\end{align*}
where $c \coloneqq n_-^1 - n_-$ and $d \coloneqq n_-^0 - n_-$.

\subsection{Connected sum and disjoint union}
\label{sub:Kh-connected-sum}

Given any two oriented links $L_1,L_2$ and any choice of connected sum, there is an exact triangle
\begin{align}
\label{eq:disjoint-skein}
\xymatrix{
\Kh(L_1 \# L_2) \ar[rr]^{\beta} && \Kh(L_1 \# -L_2) \ar[dl]^{\sigma} \\
& \Kh(L_1 \cup L_2) \ar[lu]^{\mu} & }
\end{align}
where $-L_2$ denotes reversing the orientation on all components (cf. \cite[Section 7.4]{Khovanov-Jones} and \cite[Section 3]{Rasmussen-S}).   Reduced Khovanov homology satisfies an identical triangle.

\begin{lemma}
\label{lemma:Kh-disjoint-split}
Let $L_1,L_2$ be any oriented links.  Over $\ZZ/2\ZZ$, the map $\beta$ in the skein exact triangle is identically 0 and there is a short exact sequence
\[ \xymatrix{ & 0 \ar[r] & \Kh(L_1 \# - L_2) \ar[r]^{\sigma} &\Kh(L_1 \cup L_2) \ar[r]^{\mu} &\Kh(L_1 \# L_2) \ar[r] & 0 &}\]
In particular, the merge map
\[\mu:\Kh(L_1 \cup L_2)  \rightarrow \Kh(L_1 \# L_2)\]
is surjective and the split map
\[ \sigma: \Kh(L_1 \# - L_2) \rightarrow \Kh(L_1 \cup L_2)\]
is injective.

Moreover, identical statements hold for reduced Khovanov homology.
\end{lemma}

\begin{proof}
Over $\ZZ/2\ZZ$, the Kunneth principle implies that
\[\Kh(L_1 \cup L_2) \cong \Kh(L_1) \otimes \Kh(L_2)\]
Furthermore, Lemma \ref{lemma:Kh-Kunneth} implies that
\begin{align*}
\Kh(L_1) \otimes \Kh(L_2) &\cong \Kh(L_1 \# L_2) \otimes \Kh(U) \\
& \cong \Kh(L_1 \# - L_2) \otimes \Kh(U)
\end{align*}
Consequently, since $\rk \Kh(U) = 2$, it follows that
\[\rk \Kh(L_1 \cup L_2) =  2 \cdot \rk \Kh(L_1 \# L_2) = 2 \cdot \rk \Kh(L_1 \# \widetilde{L_2})\]
Thus, this skein triple is an extremal case for the triangle inequality and the triangle unfolds to a short exact sequence.  

Over $\ZZ/2\ZZ$, we have that $\rk \Kh(L) = 2 \cdot \rk \Khr(L)$ for all links $L$ by Proposition \ref{prop:basepoint-independence}.  Therefore, the above argument can be applied {\it mutatis mutandis} to reduced Khovanov homology.
\end{proof}

Lemma \ref{lemma:Kh-disjoint-split} can be extended with arbitrary field coefficients when $L_2$ is the unlink on any number of components.

\begin{lemma}
\label{lemma:Hopf-QQ-surjective}
Let $L_1$ be an arbitrary link and let $L_2 = U_k$ denote the $k$-component unlink.  Over any field $\FF$, the map $\beta$ in the skein exact triangle is identically 0 and there is a short exact sequence
\[ \xymatrix{ & 0 \ar[r] & \Kh_{\FF}(L_1 \# - L_2) \ar[r]^{\sigma} &\Kh_{\FF}(L_1 \cup L_2) \ar[r]^{\mu} &\Kh_{\FF}(L_1 \# L_2) \ar[r] & 0 &}\]
In particular, the merge map
\[\mu:\Kh_{\FF}(L_1 \cup L_2)  \rightarrow \Kh_{\FF}(L_1 \# L_2)\]
is surjective and the split map
\[ \sigma: \Kh_{\FF}(L_1 \# - L_2) \rightarrow \Kh_{\FF}(L_1 \cup L_2)\]
is injective.
\end{lemma}

\begin{proof}
If $L_2 = U_k$, then the connected sum formulas in Lemma \ref{lemma:Kh-Kunneth} hold over every field since $L_1 \# U_k \sim L_1 \cup U_{k-1}$.  The arguments from the proof of Lemma \ref{lemma:Kh-disjoint-split} now carry over {\it mutatis mutandis}.
\end{proof}

\subsection{Mutation}
\label{sub:mutation}

Let $(B^3,t_1)$ be an abstract 2--tangle, where $t_1$ is the union of 2 arcs with endpoints on $\del B^3$ and any number of closed components.  Throughout this section, we will use $\cT$ to denote an abstract tangle and $T$ to denote a given diagram for $\cT$ in the unit disk with boundary points on the unit circle.

Suppose that $L$ and $L'$ are mutants.  We say that diagrams for the pair $L,L'$ are in {\it standard form} if they are as in Figure \ref{fig:mutation-standard}.  In particular, $L$ is obtained by connecting diagrams of two tangles $T_1,T_2$ with two bands, one untwisted and the other with a single crossing, while $L'$ is obtained by adding a single crossing to the first band and leaving the second untwisted.  Note that if the mutation is negative, the strands in the two bands are oriented in opposite directions, while if the mutation is positive they strands are oriented in the same direction.

\begin{lemma}
\label{lemma:standard}
If $L'$ is obtained from $L$ by a single mutation, then $L,L'$ admit diagrams in standard form.
\end{lemma}

\begin{proof}
Let $L$ be a link with a Conway sphere bounding the tangle $\cT_1$.  Identify the Conway sphere with the unit sphere in $\RR^3$ so that $\cT_1$ is contained in the unit ball.  We can assume that $L'$ is obtained by mutating $\cT_1$ by the involution around the $y$--axis.  The projection of $\cT_1$ to the $xy$--plane lies in the unit disk.  Isotope the exterior tangle $\cT_2$ so that its projection to the $xy$--plane is outside the unit disk.  Now, introduce a pair of canceling crossings in lower half plane, one just inside the unit disk and one just outside.  Let $T_1,T_2$ denote these diagram of $\cT_1,\cT_2$.  This gives the link on the left of Figure \ref{fig:mutation-standard}.  The corresponding diagram for $L'$ can be obtained from the diagram for $L$ by the mutation and then a flype.
\end{proof}

Take diagrams for $L,L'$ in standard form.  We can obtain 9 links from the tangle diagrams $T_1,T_2$ as follows.  Let $\cL_{\infty,\infty}$ be the link obtained by connecting $T_1$ and $T_2$ by two bands, each with a single crossing.  Each crossing has a 0-- and 1--resolution and we obtain $9$ links $\{\cL_{\bullet,\circ}\}$ for $\bullet,\circ \in \{\infty,0,1\}$ by choosing one of the three forms at each crossing.  See Figure \ref{fig:9-mutation}.

Let $N(T_1),N(T_2)$ denote the numerator closures of the 2-tangles $T_1,T_2$.  Then we have that
\begin{enumerate}
\item $\cL_{\infty,1} = L$ and $\cL_{1,\infty} = L'$
\item $\cL_{0,0} = N(T_1) \cup N(T_2)$
\item $\cL_{1,0}, \cL_{0,1}, \cL_{\infty,0}$, and $\cL_{0,\infty}$ are connected sums of $N(T_1)$ and $N(T_2)$.
\end{enumerate}

\begin{figure}
\centering
\includegraphics[width=.9\textwidth]{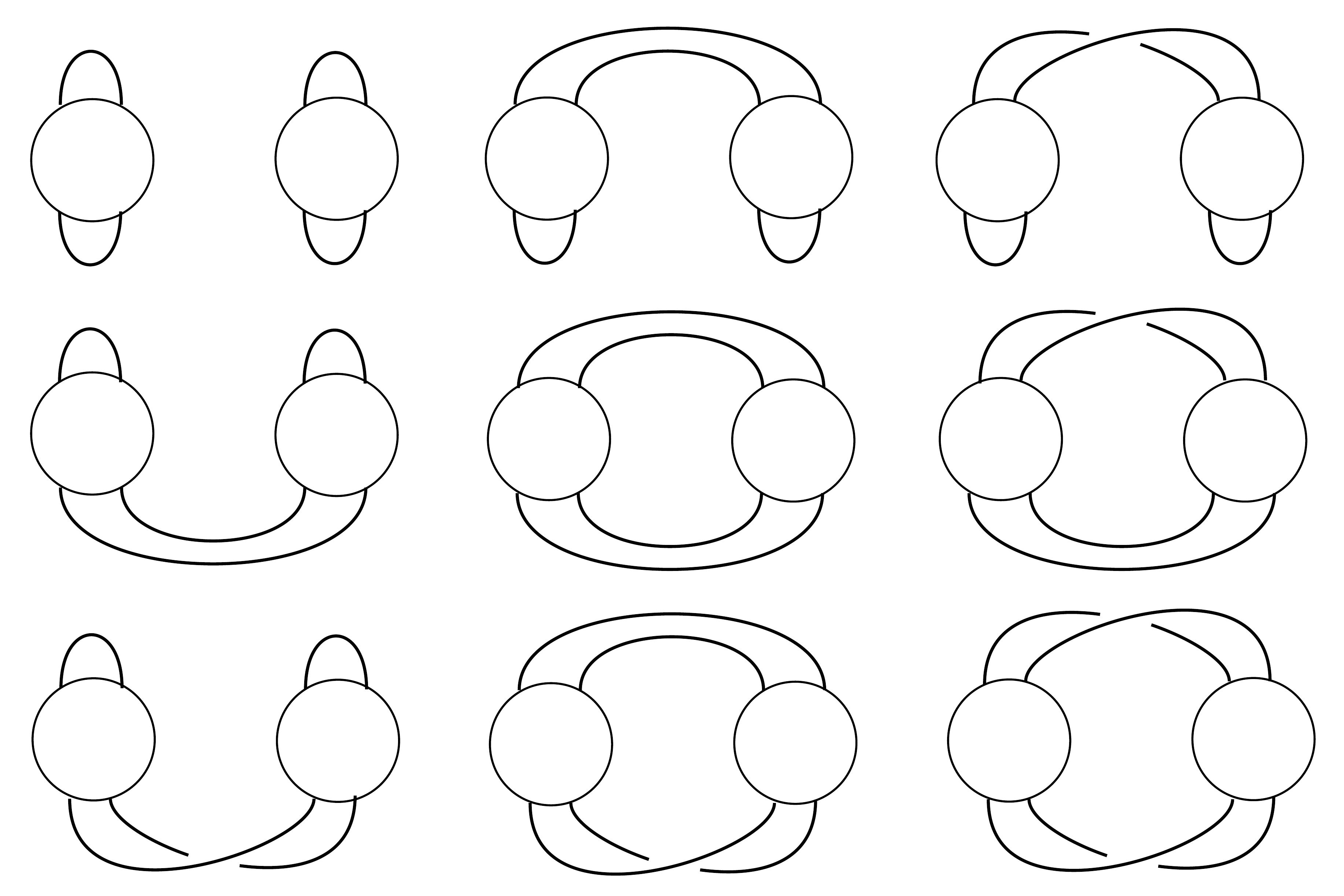}
\caption{The 9 links obtained by resolving the two crossings of $\cL_{\infty,\infty}$ (bottom right).  The links $\cL_{\infty,1}$ and $\cL_{1,\infty}$ are a mutant pair in standard form.}
\label{fig:9-mutation}
\end{figure}


\begin{figure}
\[ \xymatrix{& \ar[d] & &&\ar[d] &&& \ar[d] & \\ 
\ar[r] & \Kh(\cL_{0,0}) \ar[rrr]^{f_{0}} \ar[d]^{k_{0}} &&& \Kh(\cL_{0,1}) \ar[rrr] \ar[d]^{k_{1}} &&& \Kh(\cL_{0,\infty}) \ar[r] \ar[d] & \\
\ar[r] & \Kh(\cL_{1,0}) \ar[rrr]^{f_1} \ar[d] &&& \Kh(\cL_{1,1}) \ar[rrr] \ar[d] &&& \Kh(\cL_{1,\infty}) \ar[r] \ar[d] & \\ 
\ar[r] & \Kh(\cL_{\infty,0}) \ar[rrr] \ar[d] &&& \Kh(\cL_{\infty,1}) \ar[rrr] \ar[d] &&& \Kh(\cL_{\infty,\infty}) \ar[r] \ar[d] & 
\\ & & & & & & & & } \]
\caption{The commutative diagram of skein maps corresponding to the 9 links in Figure \ref{fig:9-mutation}.}
\label{fig:9-commutative-KFr}
\end{figure}
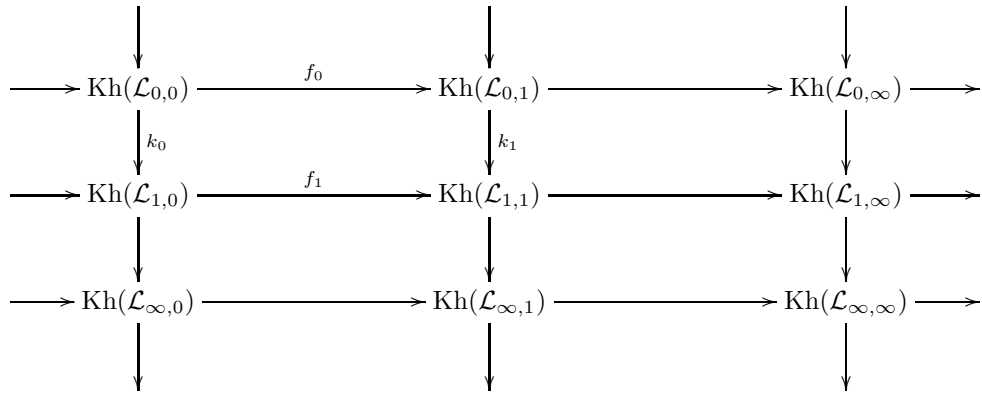

\begin{lemma}
\label{lemma:Q-closure-standard}
Suppose that $L'$ is obtained from $L$ by mutating $\cT_1$ around the $y$--axis.  For any rational closure $C$ of $\cT_1$ corresponding to this mutation, we can choose diagrams for the mutant pair $L,L'$ in standard form so that the numerator closure $N(T_1)$ is $C$.
\end{lemma}

\begin{proof}
As in the proof of Lemma \ref{lemma:standard}, choose a diagram for $L$ such that the projection of $\cT_1$ lies in the unit disk and the image of $\cT_2$ lies outside.  Let $T_1,T_2$ denote the disk diagrams for the two tangles.  If $C$ is a rational closure of $\cT_1$, then we can choose a diagram for $C$ that is the union of $T_1$ with the diagram of some rational tangle $\left[ \frac{p}{q}\right]$. This rational tangle decomposes into an annular diagram $A$ and a disk diagram $N$, with the latter containing no crossings.  The diagram $T'_1 = T_1 \cup A$ is also a projection of $\cT_1$.  Since $C$ is a rational closure corresponding to the mutation, we can assume that it is the numerator closure of $T'_1$.

Let $\overline{A}$ denote the mirror of $A$, with the convention that the outside boundary component of $A$ becomes the inside component of $\overline{A}$.  Then we obtain composite diagrams $T'_1 = T_1 \cup A$ and $T'_2 = T_2 \cup \overline{A}$ by placing the disk diagrams inside the annular diagrams.  The pair $T'_1,T'_2$ are also diagrams for $\cT_1,\cT_2$ and the union of $T'_1$ and $T'_2$ is also a diagram for $L$.  Now repeat the proof of Lemma \ref{lemma:standard} for this projection of $L$.
\end{proof}

To compute the bigrading shifts in the skein exact triangle, we need to know the relative number of positive and negative crossings among the 9 links.

\begin{lemma}
\label{lemma:pos-neg-x}
Let $n^{\pm}_{\circ,\bullet}$ for $\circ,\bullet \in \{\infty,1,0\}$ denote the number of positive and negative crossings of an oriented link diagram in Figure \ref{fig:9-mutation}.  There exist orientations on the links such that
\[ n_{1,0}^{\pm} = n_{0,1}^{\pm} = n_{0,0}^{\pm}\]
\[n_{1,\infty}^{\pm} = n_{\infty,1}^{\pm}\]
\end{lemma}

\begin{proof}
Reversing the orientation on all components and arcs of a tangle preserves the contribution of that tangle to $n^+$ and $n^-$.  Choose orientations on $N(T_1)$ and $N(T_2)$ and let $t_1^{\pm}$ and $t_2^{\pm}$ be the numbers of positive and negative crossings in the two links.  Then $n^{\pm}_{0,0} = t_1^{\pm} + t_2^{\pm}$.  Moreover, after possibly reversing the orientations on all components of $\cT_2$, this induces orientations on the connected sums of $N(T_1),N(T_2)$ and so $n^{\pm}_{1,0} = n^{\pm}_{0,1} = t_1^{\pm} + t_2^{\pm}$ as well.

If $\cL_{1,\infty}$ is a positive mutant of $\cL_{\infty,1}$, then they admit compatible orientations on the tangles and so the signs of the crossings agree.  If they are negative mutants, then given an orientation on $\cL_{1,\infty}$, we get an orientation on $\cL_{\infty,1}$ by reversing the orientation along all components in $\cT_2$.  This preserves the signs of crossings as well. 
\end{proof}

We can now prove that Khovanov homology is mutation--invariant over $\ZZ/2\ZZ$.

\begin{proof}[Proof of Theorem \ref{thrm:Kh-mutation-invariance}]

According to Lemma \ref{lemma:standard}, we can choose diagrams for $L,L'$ in standard form.   Take the 9 links obtained by the various resolutions of $\cL_{\infty,\infty}$ so that $L = \cL_{1,\infty}$ and $L' = \cL_{\infty,1}$. We obtain the commutative diagram in Figure \ref{fig:9-commutative-KFr}, where each row and column is exact.

Since we are working over a field, exactness implies that there are bigraded isomorphisms
\begin{align*}
\Kh(\cL_{1,\infty}) \cong \Ker (f_1)[-1,0] \oplus \Coker (f_1) \\
\Kh(\cL_{\infty,1}) \cong \Ker (k_1)[-1,0] \oplus \Coker (k_1)
\end{align*}
The maps $k_0$ and $f_0$ are surjective by Lemma \ref{lemma:Kh-disjoint-split} and commutativity implies that $f_1 \circ k_0 = k_1 \circ f_0$.  Thus
\[\Image (f_1) = \Image (f_1 \circ k_0) = \Image (k_1 \circ f_0) = \Image (k_1)\]
and there is a bigraded isomorphism $\Coker(f_1) \cong \Coker(k_1)$.

Finally, by Lemma \ref{lemma:pos-neg-x}, we can assume that the number of negative crossings in $\cL_{1,0}$ and $\cL_{0,1}$ agree as well as the number of negative crossings in $\cL_{1,\infty}$ and $\cL_{\infty,1}$. This implies that the bigrading shifts of $f_1$ and $k_1$ agree.  In addition, there is a bigraded isomorphism between $\Kh(\cL_{1,0})$ and $\Kh(\cL_{0,1})$.  Using the rank--nullity theorem, we can now easily find a bigraded isomorphism between $\Ker(f_1)$ and $\Ker(k_1)$.  The bigraded equivalence between $\Kh(\cL_{1,\infty})$ and $\Kh(\cL_{\infty,1})$ now follows from exactness.
\end{proof}

\begin{proof}[Proof of Theorem \ref{thrm:Kh-QQ-mutation-invariance}]
The proof follows the same argument as Theorem \ref{thrm:Kh-mutation-invariance}, except that we can use Lemma \ref{lemma:Q-closure-standard} to find diagrams for $L,L'$ in standard form so that the numerator closure of $T_1$ satisfies the hypotheses.  Moreover, we use Lemma \ref{lemma:Hopf-QQ-surjective} to guarantee that $f_0$ and $k_0$ are surjective.
\end{proof}

Watson obtained a similar but weaker result when, roughly speaking, both the mutated tangle $T$ and its complementary tangle in $L$ satify the hypothesis of Theorem \ref{thrm:Kh-QQ-mutation-invariance} \cite[Lemma 3.1]{Watson}.

\begin{proof}[Proof of Theorem \ref{thrm:KT-mutation-Khovanov}]
The proof is a directly corollary of Theorem \ref{thrm:Kh-QQ-mutation-invariance} using Lemma \ref{lemma:KT-Q-closure}.  See Subsection \ref{sub:KT}.
\end{proof}

\section{Basepoint maps}
\label{sec:basepoint}

\subsection{Knot Floer homology}
\label{sub:HFK}

A {\it pointed link} $(L,\p)$ is an oriented link $L$ in $S^3$ along with a collection $\p = (p_1,\dots,p_n)$ of basepoints along the link.  The pointed link is {\it nondegenerate} if each component of $L$ has at least one basepoint.  Let $\nu$ denote the {\it successor function} determined by the indexing of the basepoints, defined so that, following the orientation of the link, the basepoint $p_{\nu(i)}$ follows $p_i$.  Let $\SH = (\Sigma, \alphas,\betas,\zs,\ws)$ be a multi--pointed Heegaard diagram for the pointed link $(L,\p)$.  The collections of basepoints $\ws = \{w_i\}$ and $\zs = \{z_i\}$ of the Heegaard diagram are each in one--to--one correspondence with the basepoints $\p$ of the pointed link.  We assume that $z_i$ and $w_i$ lie in the same component of $\Sigma \setminus \alphas$ and $w_i$ and $z_{\nu(i)}$ lie in the same component of $\Sigma \setminus \betas$.  The multicurves $\alphas,\betas$ determine tori $\TT_{\alpha},\TT_{\beta} \subset \text{Sym}^{g + n-1}(\Sigma)$.  

Let $\CFKm(\SH)$ be the free $\FF[U_1,\dots,U_n]$--module generated by the intersection points of $\TT_{\alpha} \cap \TT_{\beta}$.  For each Whitney disk $\phi \in \pi_2(\x,\y)$, let $n_{z_i}(\phi)$ and $n_{w_i}(\phi)$ denote the multiplicity of $\phi$ at the basepoints $z_i$ and $w_i$ and let $n_{\zs}(\phi)$ and $n_{\ws}(\phi)$ denote the total multiplicity at all $z$-- and $w$--basepoints.  The chain group possesses two absolute gradings, the Alexander grading $A(\x)$ and the Maslov grading $M(\x)$.  These gradings satsify
\[M(\x) - M(\y) = \mu(\phi) - 2 n_{\ws}(\phi) \qquad \text{and} \qquad A(\x) - A(\y) = n_{\zs}(\phi) - n_{\ws}(\phi)\]
for all generators $\x,\y$ and any $\phi \in \pi_2(\x,\y)$, where $\mu(\phi)$ denotes the Maslov index of $\phi$.  The complex also possesses a diagonal grading $\delta(\x) \coloneqq M(\x) - A(\x)$.  The formal variables each satisfy $A(U_i) = -1$ and $M(U_i) = -2$.  The differential on the complex $\CFKm(\SH)$ is defined to be
\[ \del^{-}\x \coloneqq \sum_{\y \in \TT_{\alpha} \cap \TT_{\beta}} \sum_{\substack{\phi \in \pi_2(\x,\y) \\ \mu(\phi) = 1 \\ n_{\zs}(\phi) = 0}} \# \widehat{\cM}(\phi) \cdot U_1^{n_{w_1}(\phi)} \cdots U_n^{n_{w_n}(\phi)} \y\]
where $\widehat{\cM}(\phi)$ is the moduli of pseudoholomorphic representatives of $\phi$ for a fixed generic path of almost complex structures, modulo translation by $\RR$.  

For any $i = 1,\dots,n$, define the `hat' or {\it reduced knot Floer homology}
\[ \HFKh(L) \coloneqq H_*(\CFKm(\SH)/U_i = 0, \del^-)\]
It is independent of $i$ and the multi--pointed Heegaard diagram $\SH$ and is an invariant of the underlying link $L$.  The `tilde' or {\it unreduced knot Floer homology} is defined to be
\[ \HFKt(L,\p) \coloneqq H_*(\CFKm(\SH)/U_1 = \cdots = U_n = 0, \del^-)\]
It is independent of the diagram $\SH$ and an invariant of the pointed link $(L,p)$.  If $(L,p)$ is an $l$--component link with $n$ basepoints, it satisfies
\[\HFKt(L,\p) \cong \HFKh(L) \otimes V^{\otimes n - l}\]
where $V$ is a 2--dimensional vector space supported in bigradings $(0,0)$ and $(-1,-1)$.

\subsection{Grid homology}
\label{sub:grid}

Grid homology, introduced in \cite{MOS,MOST}, is one approach to constructing knot Floer homology.  The definitive reference on this material is \cite{OSS-book}.  Given a grid diagram $\GG$ of size $n$ for a link $L$, there is $2n$--pointed Heegaard diagram $\SH_{\GG} \coloneqq (T^2,\alphas,\betas,\XX,\OO)$ in which the counts $\# \widehat{\cM}(\phi)$ can be computed explicitly.  The set of generators $\St(\GG) = \TT_{\alpha} \cap \TT_{\beta}$ can be identified with the symmetric group $S_n$ and there is a one--to--one correspondence between domains which contribute to $\widetilde{\del}$ and empty rectangles $\Rect^0$ on the torus.  Let $\GCt(\GG)$ be the complex generated by $\St(\GG)$ over $\FF$ and with differential given by
\[ \widetilde{\del} \x \coloneqq \sum_{\y \in \St(\GG)} \sum_{\substack{r \in \Rect^0(\x,\y) \\ r \cap (\XX \cup \OO)= \emptyset}} \y\]
We will use $\GHt(\GG)$ when working with grid diagrams, with the understanding that it is isomorphic to the invariant $\HFKt(L,\p)$ of the corresponding pointed link.

\begin{remark}[Bigradings]
Since our explicit computations will rely on grid diagrams, we will follow the grading conventions of \cite{OSS-book}.  Specifically, these  conventions ensure that Maslov gradings of generators are integral.  For an $l$--component link, this differs by an additive shift of $\frac{l-1}{2}$ from the conventions in \cite{OS-HFK,BL-spanning,BLS} and also by multiplication by $(-1)$ from the conventions in \cite{MO-QA,Wong-skein}.  For example, our conventions dictate that for any grid diagram $\GG$ of the 2--component unlink, the grid homology satisfies
\[\GHh_m(\GG,a) \cong \begin{cases}
\FF & \text{if } (m,a) = (0,0),(-1,0) \\
0 & \text{otherwise} \end{cases} \]
The $\delta$--graded groups are defined as usual by
\[\GHt_{\delta}(\GG) \coloneqq \bigoplus_{m - a = \delta} \GHt_m(\GG,a)\]
\end{remark}

Let $L$ be an oriented link with grid diagram $\GG$.  As an ungraded vector space, the unreduced invariant  $\GHt(\GG)$ is independent of the orientations on the components of $L$.  Reversing the orientation on a component induces a well--defined shift on the $\delta$--graded group.  If $L'$ is obtained from $L$ by reversing the orientations on some components of $L$ and $\GG'$ is the corresponding grid diagram, then by \cite[Lemma 10.1.7]{OSS-book}
\[\GHt_{\delta}(\GG') \cong \GHt_{\delta + c }(\GG)\]
where $c = \frac{1}{4}\left( \text{wr}(G') - \text{wr}(G) \right)$ and $G,G'$ are the oriented planar diagrams of $L,L'$ given by the grid diagrams.  Consequently, the shifted group $\GHt(\GG)\left[\frac{1}{4}\text{wr}(G)\right]$ is independent of the link orientation.

\subsection{Basepoint maps}
\label{sub:basepoints}

Let $\SH = (\Sigma, \alphas, \betas, \zs, \ws)$ be a $2n$--pointed Heegaard diagram encoding the $n$--pointed link $(L,\p)$.  The $2n$ basepoints induce endomorphisms on $\HFKt(L,\p)$ and endow it with extra algebraic structure.  In particular, the homology group $\HFKt(L,\p)$ becomes a Clifford module over a Clifford algebra that is determined, up to isomorphism, by the partition of the $n$ basepoints among the $l$ components of $L$.  These maps, but not the Clifford module structure, have been previously considered in \cite{BL-spanning,BVV,Sarkar-basepoints,BLS,Zemke-basepoints}. 

The basepoints determine chain maps on $\CFKt(\SH)$ obtained by counting rigid disks which cross a given basepoint.  
\begin{align*}
Z_i(\x) &\coloneqq \sum_{\y \in \TT_{\alpha} \cap \TT_{\beta}} \sum_{\substack{\phi \in \pi_2(\x,\y) \\ \mu(\phi) = 1 \\ \phi \cap \zs = z_i \\\phi \cap \ws = \emptyset }} \# \widehat{\cM}(\phi) \y & W_i(\x) &\coloneqq \sum_{\y \in \TT_{\alpha} \cap \TT_{\beta}} \sum_{\substack{\phi \in \pi_2(\x,\y) \\ \mu(\phi) = 1 \\ \phi \cap \zs = \emptyset \\\phi \cap \ws = w_i }} \# \widehat{\cM}(\phi) \y
\end{align*}

Note that each $Z_i$ is homogeneous of bidegree $(-1,-1)$ and each $W_i$ is homogeneous of bidegree $(1,1)$.  Let $[x,y] = xy + yx$ denote the anticommutator.  Although the commutator and the anticommutator are equivalent mod 2, this designation will be import for extending this theory over $\ZZ$.  Then standard degeneration arguments prove that the basepoint maps satisfy the following properties.

\begin{lemma}
\label{lemma:basepoint-homotopy}
The basepoint maps $Z_i$ and $W_i$ are chain maps on $\CFKt(\SH)$ for all $i = 1,\dots,n$.  In addition, for all $1 \leq i,j \leq n$ the basepoint maps satisfy the following relations:
\begin{align*}
[Z_i,Z_j] &\sim 0 & Z^2_i \sim W^2_i &\sim 0 \\
[W_i,W_j] &\sim 0  & [Z_i,W_j] &\sim (\delta_{i,j} - \delta_{\nu(i),j})\text{Id}
\end{align*}
\end{lemma}

As a result, the basepoints determine differentials\footnote{In \cite{BL-spanning,BVV,BLS}, the symbols $\Psi_{w},\Psi_z$ and $\psi_w,\psi_z$ are used to denote the basepoint maps on the chain level and homology, respectively.}
\[z_i : \HFKt(\SH) \rightarrow \HFKt(\SH) \qquad w_i : \HFKt(\SH) \rightarrow \HFKt(\SH)\]
on the unreduced knot Floer homology for all $i = 1,\dots,n$.  

\subsection{Basepoint Clifford algebra $\Omega_{\nu}$}

Let $\Lambda_n$ be the exterior algebra of $\FF^n$.  Let $\Omega^Z_n$ be a copy of $\Lambda_n$ with the following bigrading.  Assign every element $z \in \Lambda^1(\FF^n)$ the bigrading $(-1,-1)$ and extend this bigrading multiplicatively.   Let $\Omega^W_n$ be a copy of $\Lambda(\FF^n)$ with the following bigrading.  Assign every element $w \in \Lambda^1(\FF^n)$ the bigrading $(1,1)$ and extend this bigrading multiplicatively.  

For a successor function $\nu: \{1,\dots,n\} \rightarrow \{1,\dots,n\}$, let $\Omega_{\nu}$ denote the extension of $\Omega^Z_n$ by $\Omega^W_n$ determined by the commutation relations in Lemma \ref{lemma:basepoint-homotopy}.  Specifically, let $\Omega_{\nu} = \FF[z_{1},\dots,z_{n},w_{1},\dots,w_{n}]$ modulo the relations
\begin{align*}
[z_i,z_j] &= 0 & [z_i,w_j] &= \delta_{\nu(i),j} - \delta_{i,j} \\
[w_i,w_j] &= 0 & z_i^2 &= w_i^2 = 0
\end{align*}
Extend the bigradings on $\Omega^Z_n$ and $\Omega^W_n$ to $\Omega_{\nu}$.  

Up to isomorphism, the algebra $\Omega_{\nu}$ only depends on the partition of $n$ determined by the cycles of $\nu$.  If $\phi: \{1,\dots,n\} \rightarrow \{1,\dots,n\}$, we can define a new successor function $\nu_{\phi} \coloneqq \phi^{-1} \circ \nu \circ \phi$ by conjugation.  The algebra $\Omega_{\nu_{\phi}}$ is obtained from $\Omega_{\nu}$ by relabeling its elements.  If $\nu$ and $\nu'$ induce the same partition of $n$, then we can clearly find a $\phi$ such that $\nu' = \nu_{\phi}$.

\begin{lemma}
\label{lemma:omega-clifford}
The basepoint algebra $\Omega_{\nu}$ for a successor function is the Clifford algebra $Cl(\FF^{2n},Q_{\nu})$ for a quadratic form $Q_{\nu}$ determined by $\nu$.
\end{lemma}

\begin{proof}
Fix a basis $\{z_1,w_1,\dots,z_n,w_n\}$ for $\FF^{2n}$ and let $\{x_i,y_i\}$ denote the dual basis.  Now, define $\Q_{\nu} \in \FF[x_1,y_1,\dots,x_n,y_n]$ by
\[Q_v \coloneqq \sum_{i = 1,\dots,n} (x_i y_i - x_{\nu(i)} y_i)\]
Choose some elements $v_1,\dots,v_k \in \{z_1,w_1,\dots,z_n,w_n\}$ and let $v = v_1 + \dots v_k$ denote their linear combination in $\FF\langle z_1,w_1,\dots,z_n,w_n \rangle $ $\subset \Omega_{\nu}$.  Squaring $v$, we obtain
\[v^2 = \sum_{i=1}^k v_i^2 + \sum_{1 \leq i < j \leq k} [v_i,v_j]\]
If we view $Q_v$ as a quadratic form on $\FF\langle z_1,w_1,\dots,z_n,w_n \rangle$, then it is clear from the anticommutation relations that $v^2 = Q_{\nu}(v) \text{Id}$, which is the defining relation for a Clifford algebra.
\end{proof}

  Let $\Theta \coloneqq \Lambda(\FF^2) = Cl(\FF^2,0)$, which we view as a Clifford algebra for the trivial quadratic form on $\FF^2$, and let $\Omega_1$ denote the Clifford algebra $Cl(\FF^2,Q_1)$ where $Q_1(x,y) = xy$.  By abuse of notation, we let $w,z$ denote algebra generators of both $\Theta$ and $\Omega_1$.

\begin{proposition}[Structure of $\Omega_{\nu}$]
\label{prop:omega-pi-structure}
Let $\nu$ be a successor function of length $n$ with $l$ cycles.  There is a noncanonical $\FF$--algebra isomorphism 
\[\Omega_{\nu} \cong \Omega_1^{\otimes n-l} \otimes \Theta^{\otimes l}\]
\end{proposition}

\begin{proof}
To prove the decomposition, we induct on $n - l$.

If $n - l = 0$, then $\nu(i) = i$ for all $i = 1,\dots,n$ and $Q_{\nu} = 0$.  This form has a trivial orthogonal decomposition into $n$ copies of the trivial form on $\FF^2$.  It is well--known that if a quadratic form $Q$ on $\FF^k$ has an orthogonal decomposition into $Q_1 \oplus Q_2$, where $Q_1,Q_2$ are quadratic forms on $\FF^{k_1},\FF^{k_2}$, then $Cl(\FF^k,Q)$ is the $\ZZ_2$--graded tensor product of $Cl(\FF^{k_1},Q_1)$ and $Cl(\FF^{k_2},Q_2)$.  We are working in characteristic 2 so all elements are even.  Thus
\[Cl(\FF^{2n},Q_{\nu}) = Cl(\FF^2,0) \otimes \cdots \otimes Cl(\FF^2,0) = \Theta^{\otimes n}\]

Now suppose $n - l > 0$.  We will find a successor function $\nu'$ of length $n-1$ such that
\[\Omega_{\nu} \cong \Omega_{\nu'} \otimes \Omega_1\]
Since $n - l > 0$, we can find some $i$ such that $\nu(i) \neq i$.  Without loss of generality, we can assume that $i = n-1$ and $\nu(n-1) = n$.  Define a new successor function $\nu'$ of length $n-1$ by
\[\nu'(i) \coloneqq \begin{cases}
\nu(i) & \text{if }i < n-1 \\
\nu(n) & \text{if }i = n-1
\end{cases}\]

In the bases from Lemma \ref{lemma:omega-clifford} on $\FF^{2n}$ and its dual, we can choose decompositions $V \oplus W$  and  $V^* \oplus W^*$ where
\begin{align*}
V &\coloneqq \FF \langle z_1,w_1,\dots,z_{n-2},w_{n-2},z_{n-1} + z_n,w_{n-1} + w_n \rangle & W &\coloneqq \FF \langle z_n,w_{n-1} \rangle \\
V^* &\coloneqq \FF \langle x_1,y_1,\dots,x_{n-2},y_{n-2},x_{n-1} + x_n,y_n \rangle & W^* &\coloneqq \FF \langle x_{n},y_{n-1} + y_n \rangle
\end{align*}
The elements of $W^*$ vanish on $V$ and the elements of $V^*$ vanish on $W$.  Let $X_i,Y_i$ denote the new dual basis elements.  Specifically,
\[X_{n-1} = x_{n-1} + x_n \qquad Y_n = y_{n-1} + y_n\]
and $X_i = x_i$ and $Y_i = y_i$ otherwise.  A straightforward computation shows that there is an orthogonal decomposition
\[Q_{\nu} = X_nY_n + \sum_{i=1}^{n-1} X_iY_i - X_{\nu'(i)} Y_i  = Q_1 \oplus Q_{\nu'}\]
Consequently, $ \Omega_{\nu} \cong \Omega_1 \otimes \Omega_{\nu'}$ and by induction $\Omega_{\nu} = \Omega_1 \otimes \Omega_1^{\otimes n- l - 1} \otimes \Theta^{\otimes l}$.
\end{proof}

While the decomposition of $\Omega_{\nu}$ from Proposition \ref{prop:omega-pi-structure} is noncanonical, the center of $\Omega_{\nu}$ determines a canonical copy of $\Theta^{\otimes l}$.  The successor function $\nu$ induces a shift map $\Sigma: \Omega_{\nu} \rightarrow \Omega_{\nu}$ defined by
\[\Sigma(w_i) \coloneqq w_{\nu(i)} \qquad \Sigma(z_i) \coloneqq z_{\nu(i)}\]
For each $i$, let $c_i$ be the minimum positive integer such that $\nu^{c_i}(i) = i$.  Every element of the form
\[\zeta_i \coloneqq z_i + z_{\nu(i)} + \dots + z_{\nu^{c_i-1}(i)} \qquad \omega_i \coloneqq w_i + w_{\nu(i)} + \dots w_{\nu^{c_i-1}(i)}\]
satisfies $\Sigma(v) = v$.  In fact, it is easy to see that
\[\FF \langle \zeta_1,\omega_1,\dots,\zeta_n,\omega_n \rangle = \left\{ v \in \FF^{2n} : \Sigma(v)  = v \right\}\]
This subalgebra is isomorphic to $\Theta^{\otimes l}$.

\begin{lemma}
Let $\nu$ be a successor function of length $n$.  The center of $\Omega_{\nu}$ is exactly the elements $v \in \Omega_{\nu}$ satisfying
\[\Sigma(v) = v\]
\end{lemma}

\begin{proof}
Let $z_1,w_1,\dots,z_n,w_n$ be the basis of $\FF^{2n}$ from Lemma \ref{lemma:omega-clifford} and $x_1,y_1,\dots,x_n,y_n$ the dual basis.  The commutation relations of $\Omega_{\nu}$ can equivalently be described on the basis as follows.  An arbitrary vector $v$ commutes with $z_i$ if and only if $y_{\nu^{-1}(i)}(v) = y_{i}(v)$ and commutes with $w_i$ if and only if $x_i(v) = x_{\nu(i)}(v)$.  The shift map induces a shift map $\Sigma^*$ on the dual space that satisfies $\Sigma^*x_{\nu(i)} = x_{i}$ and $\Sigma^*y_{\nu(i)} = y_{i}$.  Consequently, if $\Sigma(v) = v$ then $x_i(v) = x_{\nu(i)}(v)$ and $y_i(v) = y_{\nu(i)}(v)$ for all $i$.  Conversely, if $\Sigma(v) \neq v$ then for some $i$ either $x_i(v) \neq x_{\nu(i)}(v)$ or $y_i(v) \neq y_{\nu(i)}(v)$.  From the above discussion, it is now clear that $v$ is central if and only if $\Sigma(v) = v$.
\end{proof}

We conclude this subsection with a key fact about $\Omega_1$--modules.

\begin{lemma}
\label{lemma:omega-split}
Let $M$ be an $\Omega_1$--module.  The maps $wz$ and $zw$ are orthogonal projections that determine a direct sum decomposition
\[M  = wz  M  \oplus zw  M \]
Moreover, the following submodules are equal
\begin{align*}
wzM &= wM = \text{ker}(w) & zwM &= zM = \text{ker}(z) 
\end{align*}
and multiplication by $z$ and $w$ induce isomorphisms
\[z: wz  M  \rightarrow zw  M  \qquad w: zw  M  \rightarrow wz  M \]
\end{lemma}

\begin{proof}
The maps $wz$ and $zw$ are idempotent since
\[wzwz = wz(1 + zw) = wz \qquad zwzw = zw(1 + wz) = zw\]
and are orthogonal since $(wz)(zw) = (zw)(wz) = 0$.  The relation $wz + zw = 1$ implies that the submodules $wzM$ and $zwM$ span $M$.  This gives the direct sum decomposition.

Secondly, since $w^2 = 0$, this implies that $w  M \subset \text{ker}(w)$.  Conversely, if $w \x = 0$ then
\[ \x = (z w + wz) \x  = w z \x\]
and so 
\[ w  M \subset \text{ker}(w) \subset w z  M \subset w  M\]
Thus the three submodules are equal.  Identical arguments prove the corresponding statements for multiplication by $z$.

Finally, the isomorphisms follow from the idempotence of $wz$ and $zw$.
\end{proof}

\begin{remark}
\label{rem:rep-theory-omega-1}
We can give another interpretation of Lemma \ref{lemma:omega-split} in terms of the representation theory of $\Omega_1$.  The $\FF$--algebra $\Omega_1$ is isomorphic to the algebra $Mat(2,\FF)$ consisting of $2 \times 2$ matrices over $\FF$.  An isomorphism is given by
\begin{align*}
\label{eq:matrix-isom}
1 &\mapsto \begin{bmatrix} 1 & 0 \\ 0 & 1 \end{bmatrix} &  w &\mapsto \begin{bmatrix} 0 & 1 \\ 0 & 0 \end{bmatrix} \\
z &\mapsto \begin{bmatrix} 0 & 0 \\ 1 & 0 \end{bmatrix} &  wz &\mapsto \begin{bmatrix} 1 & 0 \\ 0 & 0 \end{bmatrix}
\end{align*}
The algebra $\Omega_1$ has a unique irreducible representation $V = \FF\langle \x,\y \rangle$ where $z\x = \y$ and $w\y = \x$.  Thus if $M$ is a finite-dimensional $\Omega_1$--module, it splits as a direct sum of several copies of $V$.  The top--degree elements in each copy of $V$ span the principal submodule $wzM$ and likewise the bottom--degree elements span $zwM$.
\end{remark}

\subsection{Clifford module structure of $\HFKt(L,\p)$}

An immediately corollary of Lemma \ref{lemma:basepoint-homotopy} is that the homology group $\HFKt(\SH)$ has the structure of a left $\Omega_{\nu}$--module. Moreover, if $(L,\p)$ is a nondegenerate, pointed, $l$--component link with a single basepoint on each component, then $\HFKt(\SH) \cong \HFKh(L)$ is a left $\Theta^{\otimes l}$--module.

Any pair of $2n$--pointed Heegaard diagrams $\SH,\SH'$ for $(L,\p)$ are related by a sequence of index 1/2 stabilizations, isotopies and handleslides.  These moves induce isomorphisms on homology and the basepoint maps commute with these isomorphisms.

\begin{proposition}
Suppose that $\SH$ and $\SH'$ are related by an index 1/2 stabilization, isotopy, or handlslide in the complement of $\ws \cup \zs$.  Then the induced isomorphism
\[\phi: \HFKt(\SH) \rightarrow \HFKt(\SH')\]
is $\Omega_{\nu}$-linear
\end{proposition}

\begin{proof}
Proved in \cite[Proposition 3.6]{BL-spanning}.
\end{proof}

Thus, the Clifford module structure is invariant up to isomorphism.  We will denote this isomorphism class by $\HFKt(L,\p)$.

Let $V$ denote a 2-dimensional bigraded vector space supported in bigradings $(0,0)$ and $(-1,-1)$.  Then if $(L,\p)$ has $n$ basepoints and $l$ components there is a noncanonical, $\FF$-linear isomorphism
\begin{equation}
\label{eq:tilde-decomp}
\HFKt(L,\p) \cong \HFKh(L) \otimes V^{\otimes n - l}
\end{equation}

The Clifford module structure gives more control over this decomposition.  In particular, a fixed decomposition $\Omega_{\nu} \cong \Theta^{ l} \otimes \Omega_1^{\otimes n- l}$, which is guaranteed by Proposition \ref{prop:omega-pi-structure}, induces a unique decomposition as in Equation \ref{eq:tilde-decomp}.  For the sake of notation, set $k = n - l$.  Since $\Omega_{\nu} \cong \Omega_1^{k} \otimes \Theta^l$, we can choose a basis $\omega_1,\zeta_1,\dots,\omega_n,\zeta_n$ for $ \Omega_{\nu}$ such that
\[[\omega_i,\omega_j] = [\zeta_i,\zeta_j] = 0 \text{ for } 1 \leq i,j \leq n
\qquad [ \omega_i,\zeta_j ] = \begin{cases}
\delta_{i,j} & \text{if } i \leq k \text{ and } j \leq k \\
0 & \text{otherwise}
\end{cases}\]
As a result, we obtain a family of $2k$ orthogonal projections $\{\omega_i\zeta_i, \zeta_i\omega_i\}$ by Lemma \ref{lemma:omega-split}.  Define  $H \coloneqq \omega_1 \zeta_1 \cdots \omega_k \zeta_k \HFKt(\SH)$.  It follows from Lemma \ref{lemma:omega-split} that there is a bigraded isomorphism
\[\HFKh(L) \cong H\]
For each $I = (i_1,\dots,i_j)$ with $1 \leq j \leq k $ and $1 \leq i_1 < i_2 < \dots < i_j \leq k$, the subspace
\[H_I \coloneqq \zeta_{i_1} \cdots \zeta_{i_j} H\]
is isomorphic to $\HFKh(L)[j,j]$.  For any two $I,I'$, the subspaces $H_I$ and $H_{I'}$ are isomorphic as (ungraded) $\Theta^{l}$--modules.



\subsection{Geometric stabilization}

The module $\HFKt(L,\p)$ is an invariant of the nondegenerate pointed link $(L,\p)$ but not the underlying link $L$ itself.  However, the Clifford module structure transforms in a well--defined way when adding or subtracting basepoints from $\p$.

Let $(L,\p')$ be a nondegenerate pointed link.  If $p$ is some point on $L$ disjoint from $\p'$ there is an $\FF$--linear isomorphism
\begin{equation}
\label{eq:tilde-stabilization}
 \HFKt(L,\p' \cup p) \cong \HFKt(L,\p') \otimes V
\end{equation}
where $V$ is a 2--dimensional vector space supported in bigradings $(0,0)$ and $(-1,-1)$.  Let $\nu'$ be the successor function of $(L,\p')$ and $\nu$ the successor function of $(L,\p' \cup p)$.  We can view $V$ as an $\Omega_1$--module (Remark \ref{rem:rep-theory-omega-1}) and therefore $\HFKt(L,\p') \otimes V$ as an $\Omega_{\nu'} \otimes \Omega_{1}$--module.  According to Proposition \ref{prop:omega-pi-structure}, we can identify the algebras $\Omega_{\nu}$ and $\Omega_{\nu'} \otimes \Omega_1$.  In fact, we can boost the isomorphism of Equation \ref{eq:tilde-stabilization} to be $\Omega_{\nu}$--linear.

\begin{proposition}
\label{prop:tilde-stabilization}
Let $(L,\p')$ be a nondegenerate pointed link with $n-1$ basepoints and successor function $\nu'$.  Let $(L,\p' \cup p)$ be the pointed link obtained by adding a single basepoint $p$, with successor function $\nu$.  Then there is an identification
\[\Omega_{\nu} \cong \Omega_{\nu'} \otimes \Omega_1\]
and an $\Omega_{\nu}$--module isomorphism
\[\HFKt(L,\p' \cup p) \cong \HFKt(L,\p') \otimes V\]
that commutes with this identification.
\end{proposition}

Adding a basepoint $p$ can be achieved on a multi--pointed Heegaard diagram $\SH$ for $(L,\p)$ by an index 0/3 stabilization.  To prove the isomorphism, we need to determine how the differential and basepoint chain maps change under such a stabilization.  

Let $\SH' = (\Sigma_g,\alphas,\betas,\zs,ws)$ be a $(2n-2)$--pointed Heegaard diagram for $L$.  In a neighborhood of $w_{n-1}$ add two new basepoints $w_n,z_n$ and two new curves $\alpha_{n-g},\beta_{n-g}$ such that $\alpha_{n-g}$ bounds a disk in $\Sigma_g$ containing $z_n$ and $w_n$ and disjoint from all other basepoints and $\beta$ curves and such that $\beta_{n-g}$ bounds a disk in $\Sigma_g$ containing $w_{n-1}$ and $z_{n}$ and no other basepoints.  We can assume that $\alpha_{n-g} \pitchfork \beta_{n-g} = \{x,y\}$, labeled so that the bigon with boundary on $\alpha_{n-g} \cup \beta_{n-g}$ and corners at $x,y$ is oriented from $x$ to $y$.  The diagram $\SH = (\Sigma_g, \alphas \cup \alpha_{n-g}, \betas \cup \beta_{n-g},\zs \cup z_n, \ws \cup w_n)$ is now a $2n$-pointed Heegaard diagram for $L$.  We say that $\SH$ is an {\it index 0/3 stabilization} of $\SH'$.

Let $C_x$ and $C_y$ denote the submodules of $\CFKt(\SH)$ spanned by generators with vertices at $x$ and $y$, respectively.  There are obvious identifications as vector spaces
\[\CFKt(\SH') \cong C_x \qquad \CFKt(\SH') \cong C_y \qquad \CFKt(\SH) \cong C_x \oplus C_y\]
Moreover, it can be shown that $C_x$ and $C_y$ are subcomplexes of $\CFKt(\SH)$ and that
\begin{align}
\label{eq:HFKt-stabilize}
\HFKt(\SH') \cong H_*(C_x) \cong H_*(C_y)[-1,-1] \qquad \HFKt(\SH) \cong \HFKt(\SH') \oplus \HFKt(\SH')[1,1]
\end{align}
where $[i,j]$ denotes shifting the bigrading \cite[Proposition 2.3]{MOS}.

\begin{proposition}
\label{prop:03-stab-basepoints}
Suppose that $\SH$ is obtained from $\SH'$ by an index 0/3 stabilization at $z_1$.  Let $W_i,Z_i$ denote the basepoint maps on $\CFKt(\SH)$ and let $W'_i,Z'_i$ denote the basepoint maps on $\CFKt(\SH')$.  For some choice of almost-complex structure, there are chain complex bijections
\[\psi_x: \CFKt(\SH') \rightarrow C_x\qquad \psi_y: \CFKt(\SH') \rightarrow C_y \]
such that the chain maps and basepoint maps satisfy the following relations
\begin{align*}
Z_i \psi_x &= \psi_x Z'_i & \text{for } i&=1,\dots,n-2  & W_i \psi_x &= \psi_x W'_i & \text{for }i &= 1,\dots,n-2 \\
Z_i \psi_y &= \psi_y Z'_i & \text{for } i&= 1,\dots,n-2 & W_i \psi_y &= \psi_y W'_i & \text{for } i&= 1,\dots,n-2\\
Z_{n-1} \psi_x &= \psi_y + \psi_x Z'_{n-1} &&& (W_{n-1} + W_{n}) \psi_x &= \psi_x W'_{n-1} \\
Z_n \psi_x &= \psi_y & & &  (W_{n-1} + W_{n}) \psi_y &= \psi_y W'_{n-1} \\
Z_{n} \psi_y = Z_{n-1} \psi_y&=0&
\end{align*}
\end{proposition}

\begin{proof}
The construction of the required almost-complex structure and a careful analysis of the relevant holomorphic disks is conducted in the proof of \cite[Proposition 6.5]{OS-HFL}.  The lemma follows by inspecting which domains cross the required basepoints.

First, we describe a correspondence between domains in $\SH$ and $\SH'$.  Let $\x,\y$ be generators of $\CFKt(\SH)$ and let $\phi \in \pi_2(\x,\y)$ be a homotopy class of Whitney disks.  This determines a 2-chain $D(\phi)$ in $\Sigma$ with boundary in $\alphas \cup \betas$.  There is a unique homotopy class $\phi' \in \pi_2(\psi_x(\x), \psi_x(\y))$ such that the 2-chain $D(\phi'_x) = D(\phi)$.  By abuse of notation, let $\phi'$ also denote the similar domain in $\pi_2(\psi_y(\x), \psi_y(\y))$ similarly.  The new alpha curve $\alpha_{g+n}$ bounds a disk $\Sigma$ that corresponds to a periodic domain $A$.  The curve $\beta_{g+n}$ bounds a disk corresponding to a periodic domain $B$.  If $n_{w_{n-1}}(\phi) > 0$, then let $\phi'_a := \phi' - n_{w_{n-1}}(\phi)\cdot A$ and $\phi'_b := \phi' - n_{w_{n-1}}(\phi)\cdot B$.  Note that every domain in $\pi_2(\psi_x(\x),\psi_x(\y))$ is a linear combination of some $\phi'$ and $A$ and $B$.

We can now summarize the counts of rigid holomorphic disks in \cite[Proposition 6.5]{OS-HFL}:
\begin{enumerate}
\item If $\phi \in \pi_2(\x,\y)$ satisfies $n_{w_{n-1}}(\phi) = 0$, then $\#\widehat{\cM}(\phi) = \#\widehat{\cM}(\phi') $.
\item If $\phi \in \pi_2(\x,\y)$ satisfies $n_{w_{n-1}}(\phi) = 1$, then $ \#\widehat{\cM}(\phi) = \#\widehat{\cM}(\phi'_a) + \#\widehat{\cM}(\phi'_b) $.
\item There exist two domains $\psi_1,\psi_2 \in \pi_2(\psi_y(\x),\psi_x(x))$ satisfying
\begin{align*}
\# \widehat{\cM}(\psi_1) &= 1 & \phi \cap \ws &= w_{n-1} & \phi \cap \zs &= \emptyset \\
\# \widehat{\cM}(\psi_2) &= 1 & \phi \cap \ws &= w_{n} & \phi \cap \zs &= \emptyset 
\end{align*}
\item There exists two domains $\psi_1,\psi_2 \in \pi_2(\psi_x(\x),\psi_y(x))$ satisfying
\begin{align*}
\# \widehat{\cM}(\psi_1) &= 1 &  \phi \cap \zs &= z_{n-1} & \phi \cap \ws &= \emptyset \\
\# \widehat{\cM}(\psi_2) &= 1 & \phi \cap \zs &= z_{n} & \phi \cap \ws &= \emptyset 
\end{align*}
\item all other domains $\phi$ with $\mu(\phi) = 1$ and $|\phi \cap (\zs \cup \ws)| \leq 1$ satisfy
\[ \#\widehat{\cM}(\phi) = 0\]
\end{enumerate}

The proposition now easily follows.
\end{proof}

With this holomorphic disk data, we can now prove the isomorphism of Proposition \ref{prop:tilde-stabilization}.

\begin{proof}[Proof of Proposition \ref{prop:tilde-stabilization}]
Let $\{w_i,z_i\}$ be the algebra generators of $\Omega_{\nu}$ and let $\{w'_i,z'_i\}$ be the algebra generators of $\Omega_{\nu'}$.  Define a map $\phi: \Omega_{\nu'} \rightarrow \Omega_{\nu}$ by
\[\phi(z'_i) \coloneqq \begin{cases} z_i & \text{if } i < n-1 \\ z_{n-1} + z_n& \text{if } i = n-1 \end{cases} \qquad \psi(w'_i)  \coloneqq \begin{cases} w_i & \text{if } i < n-1 \\ w_{n-1} + w_n & \text{if } i = n-1 \end{cases}\]
It is an injective algebra homomorphism and we can view $\HFKt(\SH)$ as an $\Omega_{\nu'}$--module by restriction of scalars.  

It follows from Proposition \ref{prop:03-stab-basepoints} that $H_*(C_x),H_*(C_y)$ are $\Omega_{\nu'}$--submodules and furthermore that the chain maps $\psi_x$ and $\psi_y$ induce $\Omega_{\nu'}$--linear isomorphisms
\[(\psi_x)*: \HFKt(\SH') \rightarrow H_*(C_x) \qquad (\psi_y)_*: \HFKt(\SH') \rightarrow H_*(C_y)\]

We can choose an identification of $\Omega_1$ with the subalgebra generated by $w_{n-1}$ and $z_n$ and also an identification
\[\Omega_{\nu} \cong \phi(\Omega_{\nu'}) \otimes \Omega_1\]
To prove the proposition, we need to find an $\Omega_{\nu'}$--linear isomorphism
\[\lambda: H_*(C_x) \oplus H_*(C_y) \rightarrow w_{n-1}z_n \cdot \HFKt(\SH) \oplus z_nw_{n-1} \cdot \HFKt(\SH)\]

First, it is clear from Proposition \ref{prop:03-stab-basepoints} that the map $w_{n-1}$ induces an isomorphism $w_{n-1}: H_*(C_y) \rightarrow H_*(C_x)$.  For dimension reasons, this implies that $H_*(C_x) = \Image(w_{n-1}) = \Image(w_{n-1}z_n)$ by Lemma \ref{lemma:omega-split}.  Secondly, consider the subspace
\[z_n w_{n-1}\cdot H_*(C_y) = (1 + w_{n-1}z_n) \cdot H_*(C_y)\]
The map $1 + w_{n-1}z_n$ is injective when restricted to $H_*(C_y)$ since $w_{n-1}z_n \x \in H_*(C_x) = \Image (w_{n-1})$ for all $\x \in H_*(C_y)$.  Again, for dimension reasons, this implies that 
\[z_n w_{n-1}\cdot  H_*(C_y) = z_n w_{n-1} \cdot \HFKt(\SH)\]
This gives an identification 
\[\HFKt(\SH) \cong w_{n-1}z_n \cdot \HFKt(\SH) \otimes V \cong \HFKt(\SH') \otimes V\]
as $\Omega_{\nu}$--modules.
\end{proof}

\subsection{Algebraic destabilization}

In practice, it is often useful to work with multi--pointed Heegaard diagrams for $L$ with many basepoints in order to compute the differential.  The cost is that the resulting homology $\HFKt(\SH) \cong \HFKh(L) \otimes V^{\otimes n - l}$ is very large.  However, the proofs of Proposition \ref{prop:omega-pi-structure} and Proposition \ref{prop:tilde-stabilization} contain a way to algebraically 'destabilize' an extra pair of basepoints without destroying any algebraic information.

Let $(L,p)$ be an $n$--pointed link with successor function $\nu$ and assume that $\nu(n-1) = n$.  Moreover, let $(L,\p \setminus {p_n})$ be the $n-1$--pointed link obtained by removing the $n^{\text{th}}$ basepoint and let $\nu'$ be its successor function.  The four basepoints $z_{n-1},w_{n-1},z_{n},w_{n}$ are successive along the link $L$ in the Heegaard diagram $\SH$.  Make the change of variables
\begin{align*}
z'_i &\coloneqq \begin{cases}
z_i & \text{if } i < n-1 \\
z_{n-1} + z_n & \text{if }i = n-1 \\
z_n & \text{if }i = n
\end{cases} &
w'_i &\coloneqq \begin{cases}
w_i & \text{if } i < n-1 \\
w_{n-1} + w_n & \text{if }i = n-1 \\
w_{n-1} & \text{if }i = n
\end{cases}
\end{align*}

As in the proof of Proposition \ref{prop:omega-pi-structure}, this determines a decomposition $\Omega_{\nu'} \otimes \Omega_1$ where the final $\Omega_1$-factor is generated by $z_n,w_{n-1}$.  From Proposition \ref{prop:tilde-stabilization}, we can conclude that there is an $\Omega_{\nu'}$--module isomorphism
\[ \HFKt(L,\p \setminus p_n) \cong w_{n-1}z_n \cdot \HFKt(L,\p)\]
More generally, we can effectively destabilize any pair of successive basepoints $z_i,w_i$ or $w_i,z_{\nu(i)}$ at the algebraic level without modifying the Heegaard diagram $\SH$.

\subsection{Orientation reversal}

The homology $\HFKt(L,p)$ is an invariant of the oriented pointed link $(L,\p)$ and reversing the orientation on a component does not preserve the bigraded invariant.  In terms of the Heegaard diagram, reversing the orientation on the component $L_i$ corresponds to swapping the $\zs$-- and $\ws$--basepoints along this component.  Thus if $\SH = (\Sigma,\alphas,\betas,\zs,\ws)$ is a multi--pointed Heegaard diagram for $(L,\p)$ and $L'$ denotes $L$ with another orientation, we can obtain a Heegaard diagram $\widetilde{\SH}' = (\Sigma,\alphas,\betas,\zs',\ws')$ for $(L',p)$ where $\ws \cup \zs = \ws' \cup \zs'$.  Thus, the chain groups $\CFKt(\SH)$ and $\CFKt(\SH')$ are identical and since $\widetilde{\del}$ ignores domains that cross basepoints, there is an ungraded isomorphism
\[\psi: \HFKt(\SH) \rightarrow \HFKt(\SH')\]
A grading computation shows that $\psi$ is homogeneous with respect to the $\delta$--grading, with a shift determined by the writhes of $L$ and $L'$ \cite{OSS-book}.

The isomorphism $\psi$ also preserves the Clifford module structure.  Suppose that $|\p| = n$ and $|L_i \cap \p| = k$.  We can label the basepoints so that the first $k$ basepoints lie on component $L_i$ and $\nu(j) = j+1 \text{ mod }k$ for $i \leq k$.  Reversing the orientation on component $i$ replaces $\nu$ with a new successor function $\nu'$ where
\[\nu'(j) = \begin{cases}
j - 1 \text{ mod }k & \text{if } j \leq k \\
\nu(j) & \text{otherwise} \end{cases} \]
There is a bijection $\phi: \ws \cup \zs \rightarrow \ws' \cup \zs'$ defined by
\[ \phi(z_j)  \coloneqq \begin{cases} 
w'_{k - i + 1} & \text{if } j \leq k \\
z'_j & \text{otherwise}
\end{cases} \qquad
\phi(w_j) \coloneqq \begin{cases} 
z'_{k - i + 1} & \text{if } j \leq k \\
w'_j & \text{otherwise}
\end{cases} \]
This induces a correspondance between the chain maps $\{W_i,Z_i\}_i$ on $\CFKt(\SH)$ and the chain maps $\{W'_i,Z'_i\}_i$ on $\CFKt(\SH')$ and an algebra isomorphism between $\Omega_{\nu}$ and $\Omega_{\nu'}$.  With this identification, it is clear that $\psi$ induces an isomorphism of Clifford modules.

\subsection{Unlinks}

The $\Theta^{\otimes l}$--module structure of $\HFKh(U_l)$ for the unlink $U_l$ is particularly simple.

\begin{lemma}
\label{lemma:basepoints-unknot}
Suppose that $(\cU_k,\p)$ is the unlink of $k$ components and a single basepoint on each component.  Let $p_i$ denote the basepoint on the $i^{\text{th}}$ component and $z_i,w_i$ the corresponding maps.  Then
\[z_i \cdot \HFKh(\cU_k) = w_i \cdot \HFKh(\cU_k) = \{0\} \]
\end{lemma}

\begin{proof}
Proved in \cite[Proposition 3.6]{BLS}.
\end{proof}

\subsection{Connected sums and disjoint unions}

The knot Floer groups satisfy a Kunneth-type formula for connected sums:
\[\HFKh(L_1 \# L_2) \cong \HFKh(L_1) \otimes \HFKh(L_2)\]
There is also a disjoin union formula, obtained from the previous isomorphism using the obvious identification $L_1 \cup L_2 \sim L_1 \# U_2 \# L_2$ where $U_2$ is the 2-component unlink.
\[\HFKh(L_1 \cup L_2) \cong \HFKh(L_1) \otimes \HFKh(U_2) \otimes \HFKh(L_2)\]
These formulas can be extended to account for the Clifford module structure.

Let $L_1,L_2$ be two oriented links.  Index the components of $L_1$ from $1$ to $l_1$ and the components of $L_2$ from $1$ to $l_2$.  For any pair $i,j$ let $L_1 \#_{i,j} L_2$ denote the link obtained by summing the $i^{\text{th}}$ component of $L_1$ to the $j^{\text{th}}$ component of $L_2$.  The homology group $\HFKh(L_1)$ is a module over $\Theta^{l_1}$.  The two basepoints on the $i^{\text{th}}$ component pick out a unique factor $\Theta_i$ of $\Theta^{l_1}$.  Similarly, $\HFKh(L_2)$ is a $\Theta^{l_2}$--module and each factor $\Theta_j$ corresponds to the $j^{\text{th}}$ component of $L_2$.  

\begin{proposition}
\label{prop:basepoints-kunneth}
Let $L_1,L_2$ be oriented links.  There is a $\Theta^{\otimes l_1 + l_2 -1 }$--linear isomorphism
\[
\HFKh(L_1 \#_{i,j} L_2) \cong \HFKh(L_1) \otimes_{\Theta_{i} = \Theta_j} \HFKh(L_2) \]
and a $\Theta^{l_1 + l_2}$--linear isomorphism
\[
\HFKh(L_1 \cup L_2)  \cong \HFKh(L_1) \otimes_{\FF} \HFKh(L_2) \otimes_{\FF} U \]
where $U$ is a 2--dimensional $\FF$--vector space supported in bigradings $(0,0)$ and $(-1,0)$.
\end{proposition}

\begin{proof}
The connect sum formula can be derived by the same arguments as the Kunneth formula description of $\CFK^{\infty}(L_1 \# L_2)$ (\cite[Theorem 7.1]{OS-HFK}, cf. \cite[Theorem 11.1]{OS-HFL}).  The only modification is to restrict to domains which cross a single basepoint.  The disjoint union formula follows from the connected sum formula.
\end{proof}

\section{The unoriented skein exact sequence}
\label{sec:skein}

The unreduced knot Floer homology groups satisfy an unoriented skein exact sequence.  Let $L_{\infty},L_{1},L_0$ be a skein triple and $-L_{\infty}, - L_1, -L_0$ their mirror images.  Manolescu  first constructed an unoriented skein exact triangle on ungraded, unreduced knot Floer homology \cite{Manolescu-skein}: 

\[ \xymatrix{
\HFKt(-L_{\infty},\p) \ar[rr]^{f_*} && \HFKt(-L_0,\p) \ar[dl]^{g_*} \\
& \HFKt(-L_1,\p) \ar[ul]^{h_*} & 
} \]

Manolescu and Ozsv{\'a}th later refined this to a exact triangle that respects $\delta$-grading \cite{MO-QA}.  Wong constructed a version of the unoriented exact triangle using grid diagrams \cite{Wong-skein}.  A different construction of the exact triangle using grid diagrams can be found in \cite{OSS-book}.  In this paper, we will use Wong's construction although we expect that the relevant results hold for any theory.

\subsection{The unoriented exact triangle via grid diagrams}

There exist compatible grid diagrams $\GG_{\infty}, \GG_1,\GG_0$ for the three links $-L_{\infty},-L_1,-L_0$, each of size $n$ and that differ only near the crossing to be resolved as in Figure \ref{fig:skein-resolutions}.  The diagrams $\GG_0,\GG_1,\GG_{\infty}$ can be realized on the same grid with three different choices $\{\gamma_0,\gamma_1,\gamma_{\infty}\}$ of the final $\beta$ curve.  The curves $\gamma_0,\gamma_1$ are included in the final pane of Figure \ref{fig:skein-resolutions}.  Wong constructs chain maps
\begin{align*}
f:& \GCt(\GG_{\infty}) \rightarrow \GCt(\GG_0) &
g:& \GCt(\GG_0) \rightarrow \GCt(\GG_1) &
h:& \GCt(\GG_1) \rightarrow \GCt(\GG_{\infty})
\end{align*}
that induce an exact triangle 
\[ \xymatrix{
\GHt(\GG_{\infty})\left[-\frac{l_{\infty}}{2}-\frac{w_{\infty}}{4}\right] \ar[rr] && \GHt(\GG_0)\left[-\frac{l_1}{2} - \frac{w_1 + 1}{4} \right] \ar[dl]^{[1]} \\
& \GHt(\GG_1)\left[-\frac{l_0}{2} - \frac{w_0-1}{4} \right] \ar[ul] & 
} \]
For $\ast \in \{\infty,0,1\}$, the constants $l_*$ and $w_*$ respectively denote the number of components of $L_*$ and the writhe of $G_*$, the planar diagram of $L_*$ given by $\GG_*$. The appropriate grading shifts are computed in \cite[Section 6]{Wong-skein}.  Translating into our present conventions, this is equivalent to the grading shifts of \cite[Theorem 10.2.4]{OSS-book}.

For our application, we will use exactness but the only explicit chain map we need is $g: \GCt(\GG_0) \rightarrow \GCt(\GG_1)$. Attaching a 1--handle induces an elementary cobordism from $-L_0$ to $-L_1$.  The map $g$ is the sum $g = \cP + \cT$ of two maps induced by this cobordism and is defined by counting holomorphic triangles.  On a grid diagram, this count is combinatorially determined. 

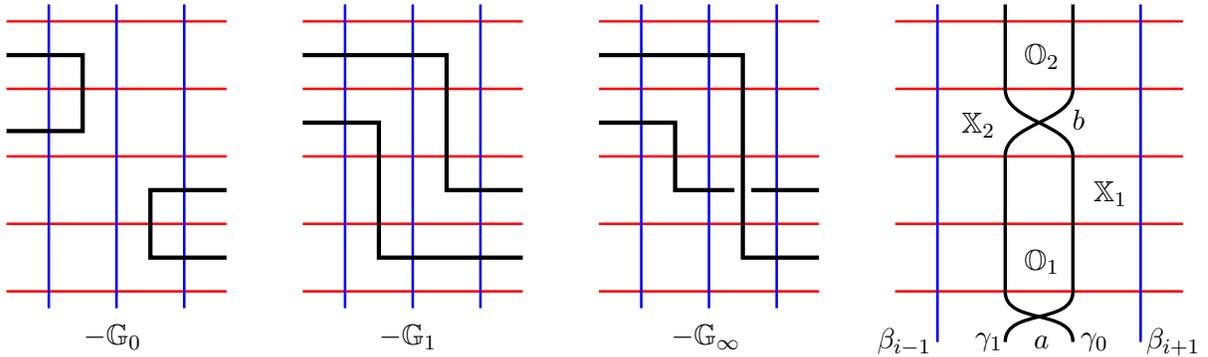
\begin{figure}
\centering
\labellist
	\large\hair 2pt
	\pinlabel $a$ at 980 0
	\pinlabel $b$ at 1015 211
	\pinlabel $\gamma_1$ at 930 0
	\pinlabel $\gamma_0$ at 1030 0
	\pinlabel $\beta_{i-1}$ at 850 0
	\pinlabel $\beta_{i+1}$ at 1105 0
	\large\hair 2pt
	\pinlabel $\OO_1$ at 980 76
	\pinlabel $\XX_1$ at 1045 141
	\pinlabel $\XX_2$ at 920 206
	\pinlabel $\OO_2$ at 980 271
	\pinlabel $-\GG_0$ at 100 5
	\pinlabel $-\GG_1$ at 380 5
	\pinlabel $-\GG_{\infty}$ at 662 5
\endlabellist
\includegraphics[width=\textwidth]{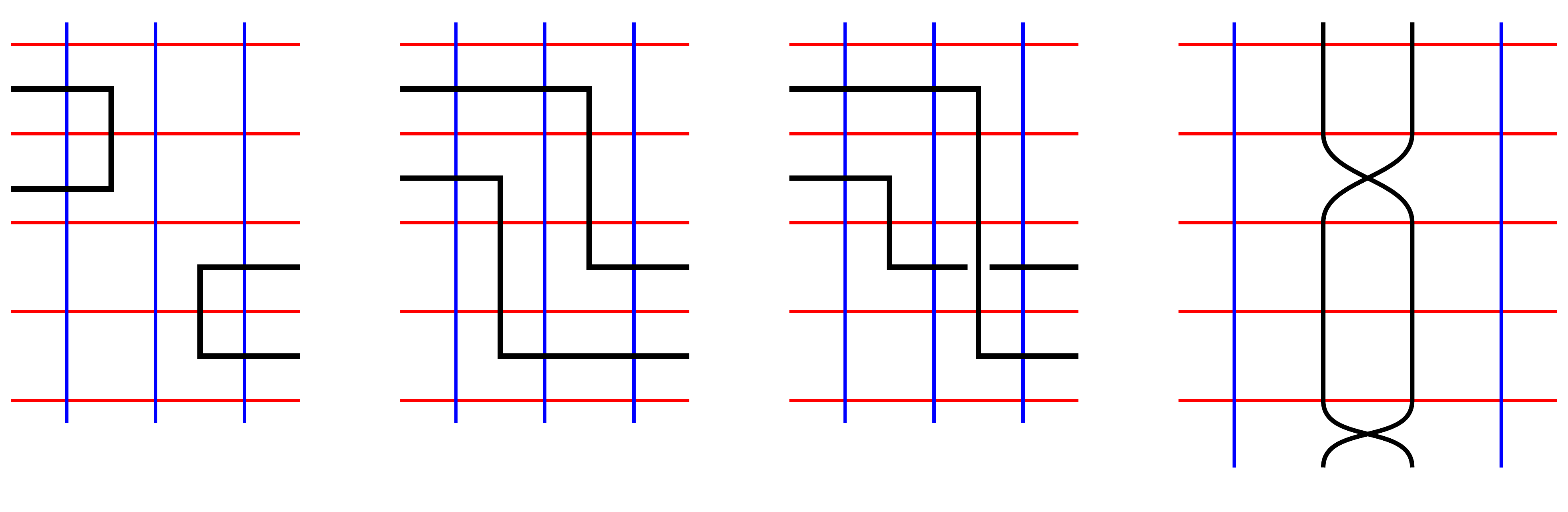}
\caption{On the left are diagrams for $-\GG_0,-\GG_1,$ and $-\GG_{\infty}$ near the crossing.  All can be obtained by rearranging the four basepoints $\XX_1,\XX_2,\OO_1,\OO_2$ on the same grid. On the right is a simultaneous diagram for $-\GG_0$ and $-\GG_1$.}
\label{fig:skein-resolutions}
\end{figure}

The chain map $g$ is obtained by counting embedded pentagons and triangles as follows. Choose $\x \in \St(\GG_{1})$ and $\y \in \St(\GG_0)$.  The two curves $\gamma_0$ and $\gamma_1$ intersect in two points $a$ and $b$.  A {\it pentagon} is an embedded polygon $p$ such that
\begin{enumerate}
\item $a \in \del p$
\item the oriented boundary of $p$ consists of 5 arcs, in order, an arc in $\gamma_0$, an arc in a horizontal circle, an arc in a vertical circle, an arc in a second horizontal circle, and an arc in $\gamma_1$.
\item all corners form angles less than $\pi$
\item viewing $\x$ and $\y$ as oriented 0--chains, the oriented boundary of the $\alpha$ arcs of $\del p$ is exactly $\y - \x$
\end{enumerate}
Let $\Pent(\x,\y)$ denote the set of pentagons from $\x$ to $\y$ and let $\Pent^0(\x,\y)$ be the set of pentagons whose interior disjoint from $\x$ and $\y$.  A {\it triangle} is defined similarly to a pentagon, except that $b \in \del p$ and the oriented boundary of $t$ consists of 3 arcs, in order, an arc in $\gamma_0$, an arc in a horizontal circle, and an arc in $\gamma_1$. Let $\Tri(\x,\y)$ denote the set of triangles from $\x$ to $\y$ and let $\Tri^0(\x,\y)$ be the set of triangles whose interior is disjoint from $\x$ and $\y$.
Define pentagon and triangle maps from $\GCt(\GG_1)$ to $\GCt(\GG_0)$ by
\begin{align*}
\cP(\x) & \coloneqq \sum_{\y \in \St(\GG_1)} \sum_{\substack{p \in \Pent^0(\x,\y) \\ p \cap \OO = 0 \\ p \cap \XX = 0}} \y &
\cT(\x) & \coloneqq \sum_{\y \in \St(\GG_1)} \sum_{\substack{p \in \Tri^0(\x,\y) \\ p \cap \OO = 0 \\ p \cap \XX = 0}} \y 
\end{align*}
and set $g: = \cP + \cT$.  Standard degeneration arguments show that $\cP,\cT$ and $g$ are chain maps.

\subsection{Equivariance of $g_*$}
\label{sub:basepoint-equivariance}

The induced homology maps $\cP_*,\cT_*$ are mostly, but not fully, equivariant with respect to the basepoint maps.  This is required because the homology groups are Clifford modules over different Clifford algebras.  The diagrammatic 1--handle attachment changes the successor function $\nu$ and therefore the ring $\Omega_{\nu}$ .  However, we can exactly quantify the failure of $\cP_*, \cT_*$ and $g_*$ to be equivariant.  

For simplicity, assume that the saddle map from $L_1$ to $L_0$ is orientable.  We can label the basepoints as $\XX$ and $\OO$ in $\GG_1,\GG_0$ near the saddle move as in Figure \ref{fig:skein-resolutions}, so that for each $i$, the basepoints $\XX_i$ and $\OO_i$ are sequential along the link $L_1$.  The nonorientable case is similar, except that the partitions of the basepoints between $\XX$ and $\OO$ are different.  In $\GG_0$, we relabel $\XX_2$ as an $\OO$--basepoint and $\OO_2$ as an $\XX$-basepoint.  In $\GG_1$, we relabel $\XX_2$ as an $\OO$--basepoint and $\OO_1$ as an $\XX$--basepoint. 
The following relations are a straightforward extension of \cite[Propositions 3.7, 3.8]{BLS}.

\begin{lemma}
Suppose that $L_0,L_1$ are oriented as in Figure \ref{fig:skein-resolutions}.  Let $W_i$ denote the basepoint map on $\GHt(\GG_1),\GHt(\GG_0)$ induced by $\OO_i$ and let $Z_i$ denote the basepoint map on $\GHt(\GG_1),\GHt(\GG_0)$ induced by $\XX_i$.  Then the following maps are homotopic:
\begin{align*}
[\cP,W_i] & \sim 0 & \text{for } i&=1,\dots,n & [\cP,Z_i] & \sim 0 &  \text{for } i&=3,\dots,n \\
[\cT,W_i] & \sim 0 &  \text{for } i&=3,\dots,n & [\cT,Z_i] & \sim 0 &  \text{for } i&=1,\dots,n \\
[\cT,W_i] & \sim \cP &  \text{for } i&=1,2  & [\cP,Z_i] & \sim \cT &  \text{for } i&=1,2
\end{align*}
If the cobordism is not orientable, the same relations hold after relabeling the basepoint maps.
\end{lemma}
\begin{proof}
The required homotopies are obtained counting pentagon and triangle maps that cross the appropriate basepoints.  They are defined as
\begin{align*}
\cP_{\XX_i}(\x) & \coloneqq \sum_{\y \in \St(\GG_1)} \sum_{\substack{p \in \Pent^0(\x,\y) \\ p \cap \OO = 0 \\ p \cap \XX_j = 0 \text{ if } i\neq j \\ \XX_i \in p}} \y  &
\cP_{\OO_i}(\x) & \coloneqq \sum_{\y \in \St(\GG_1)} \sum_{\substack{p \in \Pent^0(\x,\y) \\ p \cap \XX = 0 \\ p \cap \OO_j = 0 \text{ if } i\neq j \\ \OO_i \in p}} \y \\
\cT_{\XX_i}(\x) & \coloneqq \sum_{\y \in \St(\GG_1)} \sum_{\substack{p \in \Tri^0(\x,\y) \\ p \cap \OO = 0 \\ p \cap \XX_j = 0 \text{ if } i\neq j \\ \XX_i \in p}} \y & 
\cT_{\OO_i}(\x) & \coloneqq \sum_{\y \in \St(\GG_1)} \sum_{\substack{p \in \Tri^0(\x,\y) \\ p \cap \XX = 0 \\ p \cap \OO_j = 0 \text{ if } i\neq j \\ \OO_i \in p}} \y
\end{align*}

In the first four cases, each term in the appropriate equation is obtained by composing a rectangle and a pentagon or a rectangle and a triangle and each composite domain has exactly 2 decompositions of this form.  For the final two cases, however, there is an extra degeneration case to consider.  This stems from the fact that there are two bigons $B_1,B_2$ from $b$ to $a$, where $B_i$ contains $\OO_i$, and two annuli $A_1,A_2$ from $b$ to $a$, where $A_i$ contains $\XX_i$.  The composite of a triangle and a rectangle can decompose into some $B_i$ and an empty pentagon, while the composite of a pentagon and a rectangle can decompose into some $A_i$ and an empty triangle.

In the nonorientable case, relabeling the basepoints clearly does not affect the degeneration arguments.
\end{proof}

Passing to homology, we obtain the following proposition as an easy corollary of the previous lemma.

\begin{proposition}
\label{prop:basepoint-equivariance}
Let $w_i,z_i$ denote the basepoint maps on $\GHt(\GG_0)$ and $\GHt(\GG_1)$.  If the cobordism is orientable, then the basepoint maps satisfy the following anticommutation relations with $g_*$
\begin{align*}
[g_*,z_i] & = 0 &  \text{for } i&=3,\dots,n  & [g_*,z_1 + z_2] &= 0 \\
[g_*,w_i] & = 0 &  \text{for } i&=3,\dots,n & [g_*,w_1+w_2] &= 0 \\
[g_*,z_i + w_j] &= g_*  & \text{for } i,j& \in \{1,2\}
\end{align*}
If the cobordism is not orientable, the same relations hold after relabeling the basepoint maps.
\end{proposition}

\subsection{Connected sum and disjoint union}
\label{sub:disjoint-connect}

Let $L_1,L_2$ be oriented links and let $-L_2$ denote $L_2$ with the opposite orientation.  As in Subsection \ref{sub:Kh-connected-sum}, we consider the skein triple  \[L_1 \# -L_2 \qquad  L_1 \cup L_2 \qquad  L_1 \# L_2 \]
Choose grid diagrams $\GG_{\infty},\GG_0,\GG_1$ for these three links with the required form near the crossing to be resolved.  The grid homologies satisfy
\begin{align*}
\GHt(\GG_0) &\cong \HFKh(L_1) \otimes \HFKh(L_2) \otimes V^{n-l} \\
\GHt(\GG_{\infty}) &\cong \HFKh(L_1) \otimes \HFKh(L_2) \otimes V^{n-l} \\
\GHt(\GG_1) &\cong \HFKh(L_1) \otimes \HFKh(L_2) \otimes \HFKh(\mathcal{U}_2) \otimes V^{n - l - 1}
\end{align*}
Thus, all three groups have the same total rank over $\FF$.  However, $V$ and $\HFKh(\cU_2)$ can be distinguished by their $\delta$--gradings.  This allows us to compute the $\delta$--graded ranks of the maps in the skein exact triangle.  The number of components satisfies
\[l_1 = l_0 + 1 = l_{\infty} + 1\]
and the writhes satisfy
\[\text{wr}(G_{\infty})  = \text{wr}(G_1) -1 = \text{wr}(G_0) - 1\]
since $G_{\infty}$ has exactly 1 extra negative crossing.  The graded long exact sequence is therefore

\begin{align*}
 \xymatrix{
\ar[r] & \GHt_{\delta} (\GG_0) \ar[r]^{g_*} & \GHt_{\delta}(\GG_1) \ar[r]^{h_*} & \GHt_{\delta}(\GG_{\infty}) \ar[r]^{f_*} & \GHt_{\delta - 1}(\GG_0) \ar[r] &
}
\end{align*}

\begin{lemma}
\label{lemma:half-rank}
For every $\delta \in \ZZ$, the rank of $g_*: \GHt_{\delta}(\GG_0) \rightarrow \GHt_{\delta}(\GG_1)$ is equal to $\frac{1}{2} \rk \GHt_{\delta}(\GG_1)$.
\end{lemma}

\begin{proof}
First, note that since
\[ \rk \GHt(\GG_{\infty}) = \rk \GHt(\GG_{1}) =\rk \GHt(\GG_{0})  = 2^{n-l} \cdot \rk \HFKh(L_1) \cdot \rk \HFKh(L_2)\]
exactness implies that the total ranks of $f_*,g_*,h_*$ are all equal to $2^{n-l-1} \cdot \rk\HFKh(L_1) \cdot \rk \HFKh(L_2)$.  The statement in the lemma refines this fact to the level of $\delta$ gradings.

To prove the lemma, we will prove the stronger result that for any $\delta \in \ZZ$, the maps
\begin{align*}
g_*: & \GHt_{\delta}(\GG_{0}) \rightarrow \GHt_{\delta}(\GG_{1})\\
h_*: & \GHt_{\delta}(\GG_{1}) \rightarrow \GHt_{\delta}(\GG_{\infty})\\
f_*: & \GHt_{\delta}(\GG_{\infty}) \rightarrow \GHt_{\delta-1}(\GG_{0})
\end{align*}
all have rank $\frac{1}{2} \rk \GHt_{\delta} (\GG_1)$.

First, the fact is clear for $\delta \gg 0$ since the knot Floer homology groups are bounded.  Thus 
\[\GHt_{\delta}(\GG_{\infty}) = \GHt_{\delta}(\GG_{0}) = \GHt_{\delta}(\GG_{1}) = 0\]
and the ranks of the maps are all equal to $\frac{1}{2}0 = 0$.  

Now, suppose that the statement is true for $\delta + 1$.  We have that
\begin{align*}
\rk \, \GHt_{\delta}(\GG_0) &= \frac{1}{2} \left( \rk \GHt_{\delta+1}(\GG_1) + \rk \GHt_{\delta}(\GG_1)\right) \\
\rk \, \GHt_{\delta}(\GG_1) &= \rk \GHt_{\delta}(\GG_1) \\
\rk \, \GHt_{\delta}(\GG_{\infty}) &= \rk \GHt_{\delta}(\GG_1)
\end{align*}

Thus, the summand of $\Coker(f_*)$ of grading $\delta$ has rank $\frac{1}{2} \rk \GHt_{\delta}(\GG_1)$ and by exactness, this is the rank of $g_*$ restricted to $\GHt_{\delta}(\GG_0)$.  Repeating this argument twice, we see that the same statement holds for $h_*$ and $f_*$.  Proceeding by induction, this proves the statement for all $\delta \in \ZZ$.
\end{proof}

\section{$\HFKh_{\delta}$ and mutation}
\label{sec:HFK}

\subsection{Setup}
\label{sub:mutation-HFKt}

Suppose that $L$ is the union of the 2--tangles $\cT_1,\cT_2$ and that $L'$ is obtained from $L$ by mutating $\cT_1$.  For any rational closure $C$ of $\cT_1$, Lemma \ref{lemma:Q-closure-standard} states that we can find diagrams for $L,L'$ in standard form so that the numerator closure $N(T_1)$ is $C$.  As in Section \ref{sec:KF}, let $\cL_{\infty,\infty}$ be the link obtained by connecting the diagrams $T_1$ and $T_2$ by bands with a single twist.  Resolving each of the two crossings give a collection of nine links $\{\cL_{\bullet,\circ}\}$ for $\bullet,\circ \in \{\infty,0,1\}$.  Let $l_{\bullet,\circ}$ denote the number of components of $\cL_{\bullet,\circ}$. We can approximate these link diagrams with grid diagrams to obtain a collection of nine grid diagrams $\{\GG_{\bullet,\circ}\}$ that agree except near the crossings and have the form required for the skein exact triangle.  Taking grid homology and considering the skein exact triangle, we obtain the commutative diagram of Figure \ref{fig:9-commutative-GHt} with each row and each column exact.

\begin{figure}
\[\xymatrix{& \ar[d] &\ar[d] & \ar[d] & \\ 
\ar[r] & \GHt(\GG_{0,0}) \ar[r]^{f_{0}} \ar[d]^{k_{0}} & \GHt(\GG_{0,1}) \ar[r] \ar[d]^{k_{1}} & \GHt(\GG_{0,\infty}) \ar[r] \ar[d] & \\
\ar[r] & \GHt(\GG_{1,0}) \ar[r]^{f_1} \ar[d] & \GHt(\GG_{1,1}) \ar[r] \ar[d] & \GHt(\GG_{1,\infty}) \ar[r] \ar[d] & \\ 
\ar[r] & \GHt(\GG_{\infty,0}) \ar[r] \ar[d] & \GHt(\GG_{\infty,1}) \ar[r] \ar[d] & \GHt(\GG_{\infty,\infty}) \ar[r] \ar[d] & 
\\ & & & & }\]
\caption{The commutative diagram of skein maps for $\GHt$ corresponding to the 9 links in Figure \ref{fig:9-mutation}.}
\label{fig:9-commutative-GHt}
\end{figure}

To prove Theorem \ref{thrm:HFK-mutation-invariance}, we will attempt to mimic the proof of Theorem \ref{thrm:Kh-mutation-invariance} in Section \ref{sec:Kh}.  As for Khovanov homology over $\ZZ/2\ZZ$, there is a graded isomorphism between $\GHt(\GG_{1,0})$ and $\GHt(\GG_{0,1})$ since they are connected sums of the same two pointed links. Mutation--invariance of $\HFKh_{\delta}$ will follow if $f_1$ and $k_1$ have the same $\delta$--graded ranks.    However, knot Floer homology and Khovanov homology over $\ZZ/2\ZZ$ behave differently with respect to merge map
\[\mu: L_1 \cup L_2 \rightarrow L_1 \# L_2\]
(compare Lemma \ref{lemma:half-rank} with Lemma \ref{lemma:Kh-disjoint-split}).  Given the assumption on rational closures, we can use the extra algebraic structure given by the basepoint maps to work around this fact.

\subsection{Basepoint structure}  We will review the relevant results of Sections \ref{sec:basepoint} and \ref{sec:skein} in light of the setup in the previous subsection.

First, we make some notational remarks.  If $l_{0,0} \equiv l_{1,1} \text{ mod }2$, then we can choose compatible orientations on $\cL_{0,0},\cL_{1,0},\cL_{0,1},\cL_{1,1}$ so that the elementary cobordisms among these four links are oriented.  The grid diagrams $\GG_{0,0},\GG_{1,0},\GG_{0,1},\GG_{1,1}$ are all realized on the same grid but different $\beta$--curves.  Thus, we can speak unambiguously about the basepoint collections $\XX$ and $\OO$ on all four grids.  Moreover, we will use the the same notation $\{w_i,z_j\}$ to denote the basepoint maps as operators on the four homology groups.  It should be pointed out, however, that the anticommutation relations among these maps are unique for each of the four homology groups.  

If $l_{0,0} = l_{1,1} + 1$, however, we will adopt the following convention.  Here, the cobordisms to $\cL_{1,1}$ are nonorientable.  The complication in this case is fundamentally a problem of {\it labeling} the basepoint maps and not their algebraic structure.  All four diagrams lie on the same grid with the same basepoints, except that we cannot take the same partition into $\XX$ and $\OO$ on all four diagrams.  Our convention is to fix an orientation on $\cL_{0,0}$, which determines a partition into $\XX$ and $\OO$ and a labeling of the basepoint maps as $z_i$ or $w_j$ accordingly, then keep this labeling on the remaining three homology groups.  For example, if $\OO_1$ on $\GG_{0,0}$ becomes an $\XX$-basepoint in $\GG_{1,1}$, we will still use $w_1$ to denote the map which counts disk that cross this point on the grid.

Second, let $w_T,z_T$ denote the sums of all $\w$ and $\z$ basepoint maps on $N(T_1)$, respectively.  Using our notation conventions, let it also denote the corresponding sum of the same basepoints on any of the 8 remaining $\GHt$ groups.  Note the the endomorphism $w_T + z_T$, as an operator on the ungraded $\GHt$ group, is independent of the choice of orientations.  Moreover, the combination of Lemma \ref{lemma:basepoints-unknot} and Proposition \ref{prop:basepoints-kunneth} implies that 
\begin{align}
(w_T + z_T) \cdot \GHt(\GG_{0,0}) = \{0\}
\end{align}
if $N(T_1)$ is the unlink.  

Third, we describe the Clifford module structures on the homology groups.  In particular, the skein maps affect the successor function so the various homology groups are modules over different Clifford algebras.  Below, we state explicitly the relevant commutation relations we will need. 

We can always choose orientations so that the elementary cobordism $\GG_{0,0} \rightarrow \GG_{1,0}$ is oriented.  Let $w_1,z_1$ denote the basepoints on $N(T_1)$ and let $w_2,z_2$ denote the basepoints on $N(T_2)$ nearest this handle attachment.  Let $w_3,z_3$ and $w_4,z_4$ denote the basepoints on $N(T_1)$ and $N(T_2)$, respectively, near the second handle attachment.  

After possibly stabilizing the grid diagram, we can assume that as operators on $\GHt(\GG_{0,0})$ the basepoint maps satisfy
\begin{align}
 [w_i,z_j] &= \delta_{i,j} & \text{for }&1 \leq i,j \leq 4
\end{align}
and furthermore that
\begin{align}
[w_T + z_T,w_i] &= [w_T + z_T,z_i] = 0 & \text{for all } i&=1,\dots,n
\end{align}

On $\cL_{1,0}$, the basepoints $w_1,z_2$ are consecutive.  As operators, the basepoint maps satisfy the following relations:
\begin{align}
\label{eq:comm-relation}
\begin{aligned}
 \leftbracket w_1,z_2] &= 1 & [w_T,w_1] &= 0 &[z_T,z_1] &= 0  \\
[w_T,z_2] &= 1 & [z_T,w_1] &= 1 &&\\
[w_T + z_T,w_1z_2] &= (w_1 + z_2)  &   [w_T + z_T,z_2w_1] &= (w_1 + z_2) &&
\end{aligned}
\end{align}

Next, we can choose an orientation on $\cL_{1,1}$ so that it agrees with the orientation on the segment of $\cL_{1,0}$ containing $w_1$ and $z_2$.   These are still consecutive so $[w_1,z_2] = 1$ when acting on $\GHt(\GG_{1,1})$.

We could also choose orientations so that the elementary cobordism $\cL_{0,0} \rightarrow \cL_{0,1}$ is oriented and partition the basepoints accordingly.  This will give the same relations as above for the basepoints $w_3,z_3,w_4,z_4$ instead.  Moreover, we can choose an orientation on $\cL_{1,1}$ to agree with the orientation of the segment of $\cL_{0,1}$ containing $w_3$ and $z_4$.

Finally, we state the important pseudoequivariance properties of the skein and basepoint maps.  First, exactly 1 $w$-- and 1 $z$--basepoint on $N(T_1)$ is adjacent to each of the 1--handle attachments.  Thus
\begin{align}
\label{eq:map-comm}
\leftbracket f_i,w_T + z_T] = f_i \qquad [k_i, w_T + z_T] = k_i \qquad \text{for }i=0,1
\end{align}
Secondly, the basepoints $w_1,z_2$ are away from the 1--handle that induces $f_0,f_1$.  Thus, the corresponding maps commute with the induced map on homology.  Similarly, the basepoints $w_3,z_4$ are away from the 1--handle that induces $k_0,k_1$.  
\begin{align}
\label{eq:comm-relation-2}
\leftbracket f_1,w_1] = [f_1,z_2] = 0 \qquad [k_1,w_3] = [k_1,z_4] = 0
\end{align}

\subsection{Virtual surjectivity}

While the merge maps $f_0,k_0$ are not surjective, they are sufficiently close to being surjective for the purposes of proving Theorem \ref{thrm:HFK-mutation-invariance}.  Specifically, their images generate the codomains over the appropriate basepoint algebra.

\begin{lemma}
\label{lemma:virtually-surjective}
Suppose that $N(T_1)$ is the unlink. 
Then 
\begin{align*}
\GHt(\cL_{1,0},\p) &= w_1z_2 \cdot \Image(k_0) + z_2w_1 \cdot \Image(k_0) \\
\GHt(\cL_{0,1},\p) &= w_3z_4 \cdot \Image(f_0) + z_4w_3 \cdot \Image(f_0)
\end{align*}
\end{lemma}

\begin{proof}
We will only prove the first statement; the second follows by an identical argument.   

The basepoints $w_1,z_2$ are consecutive in $\GG_{1,0}$, so Lemma \ref{lemma:omega-split} states that the orthogonal projections $w_1z_2$ and $z_2w_1$ determine a direct sum decomposition of $\GHt(\GG_{1,0})$.  Each summand has rank $\frac{1}{2} \rk \GHt(\GG_{1,0})$, which is also the rank of $\Image(k_0)$ by Lemma \ref{lemma:half-rank}.  This rank calculation implies that $w_1z_2$ and $z_2w_1$ are surjective if and only if they are injective.  The first equality of the lemma follows immediately if the projections $w_1z_2$ and $z_2w_1$, restricted to $\Image (k_0)$, are surjective.  As a result, it suffices to prove that if $w_1z_1 \cdot k_0(\x) = 0$ or $z_2w_1 \cdot k_0(\x) = 0$ for some $\x \in \GHt(\GG_{0,0})$ then $k_0(\x) = 0$.

Suppose that $w_1z_2 \cdot k_0(\x) = 0$.  Then applying the anticommutation relations we obtain:
\begin{align}
0 &= w_1 z_2 \cdot k_0(\x) \nonumber \\
&= (1 + z_T + w_T) \cdot w_1 z_2 \cdot k_0(\x) \nonumber \\
&= w_1 z_2 \cdot  (1 + z_T + w_T)\cdot k_0(\x) + (w_1 + z_2) \cdot k_0(\x) \nonumber \\
&= w_1z_2 \cdot k_0( (z_T + w_T)\cdot\x) + (w_1 + z_2) \cdot k_0(\x) \nonumber \\
&= (w_1 + z_2) \cdot k_0(\x)
\label{eq:2-vanish}
\end{align}

Equation \ref{eq:2-vanish} follows from the previous equality since $(w_T + z_T)$ is identically 0 as an operator on $\GHt(\GG_{0,0})$.  Since $w_1$ and $z_2$ are consecutive on $\GG_{1,0}$, we have that $(w_1 + z_2)^2 = [w_1,z_2] = 1$ and so the linear operator $(w_1 + z_2)$ is invertible and therefore injective.  Applying this to Equation \ref{eq:2-vanish} implies that $k_0(\x) = 0$.  A similar argument proves that if $z_2 w_1 \cdot k_0(\x) = 0$, then $k_0(\x) = 0$ as well.
\end{proof}

\begin{remark}
While Lemma \ref{lemma:half-rank} implies $f_0$ and $k_0$ are never honestly surjective, we can interpret Lemma \ref{lemma:virtually-surjective} as saying that $f_0$ and $k_0$ are `virtually' surjective in the following sense.  

Let $\Omega'$ denote the common subring of $\Omega_{\nu_{0,0}}$ and $\Omega_{\nu_{1,0}}$ generated by the elements $z_1 + z_2,w_1 + w_2,z_3,w_3,\dots,z_n,w_n$.  If we view $\GHt(\GG_{0,0})$ and $\GHt(\GG_{1,0})$ as $\Omega'$--modules by restriction of scalars, then it follows from Proposition \ref{prop:basepoint-equivariance} that $f_0$ is $\Omega'$-linear.

There is an identification $\Omega_{\nu_{1,0}} \cong \Omega' \otimes_{\FF} \Omega_1$.  By extending scalars, we can therefore view $\GHt(\GG_{0,0}) \otimes_{\Omega'} \Omega_1$ as an $\Omega_{\nu_{1,0}}$--module and extend $f_0$ to an $\Omega_{\nu_{1,0}}$--linear map
\[ f_0: \GHt(\GG_{0,0}) \otimes_{\Omega'} \Omega_1 \rightarrow \GHt(\GG_{1,0})\]
Lemma \ref{lemma:virtually-surjective} says that this is a surjective map of $\Omega_{\nu_{1,0}}$--modules.
\end{remark}

\subsection{Rank equivalence}  We can now finish the proof of Theorem \ref{thrm:HFK-mutation-invariance} by equating the graded ranks of $f_1$ and $k_1$.

The skein maps are homogeneous with respect to the $\delta$-grading, thus their images and kernels can be decomposed into summands that are homogeneous with respect to the $\delta$-grading.  For each $\delta \in \ZZ$, let $\Image_{\delta}$ and $\ker_{\delta}$ denote the summand in grading $\delta$.

\begin{lemma}
\label{lemma:double-rank}
Suppose that $N(T_1)$ is the unlink.  Then
\[
\rk \Image_{\delta} (f_1) = 2 \cdot \rk \Image_{\delta} (f_1 \circ k_0)  = 2 \cdot \rk \Image_{\delta} (k_1 \circ f_0) = \rk \Image_{\delta} (k_1)
\]
\end{lemma}

\begin{proof}
The middle equality of the lemma follows from commutativity of the diagram. The first and third equalities have identical proofs.  We will thus focus on the first equality.  Moreover, the graded statement follows easily from the corresponding ungraded statement.  Thus we will ignore the $\delta$ gradings.

Let $K$ denote the kernel of $f_1$ restricted to $\Image (k_0)$.  From the proof of Lemma \ref{lemma:virtually-surjective}, we know that the projection maps $w_1z_2$ and $z_2w_1$ are injective.  Thus, the first equality of the lemma would follow if 
\begin{align}
\label{eq:kernel-condition}
\ker (f_1) = w_1 z_2 \cdot K \oplus z_2 w_1 \cdot K
\end{align}
Equation \ref{eq:kernel-condition} implies that $\rk \ker (f_1) = 2 \cdot \rk K$ and so the ranks of the images of $f_1$ and $f_1 \circ k_0$ must satisfy the same relation.

The maps $w_1,z_2$ commute with $f_1$, so the projections $w_1z_2$ and $z_2w_1$ determine orthogonal decompositions of both $\GHt(\GG_{1,0})$ and $\GHt(\GG_{1,1})$ that commute with $f_1$.  Consequently, if $f_1(k_0(\x)) = 0$, then
\begin{align*}
f_1(w_1z_2  \cdot k_0(\x))& = w_1z_2 \cdot f_1(k_0(\x)) = 0 \\
 f_1(z_2w_1 \cdot k_0(\x)) &= z_2 w_1 \cdot f_1(k_0(\x)) = 0
\end{align*}
This implies that $\ker(f_1)$ contains $w_1z_2 \cdot K \oplus z_2w_1 \cdot K$.  

To finish the proof of Equation \ref{eq:kernel-condition}, we need to show the reverse inclusion.  Take $\y \in \GHt(\GG_{1,0})$ and suppose $f(\y) = 0$.  By Lemma \ref{lemma:virtually-surjective}, we can express $\y$ as $\y = w_1z_2 \cdot k_0(\x_1) + z_2 w_1 \cdot k_0(\x_2)$ for some $\x_1,\x_2 \in \GHt(\GG_{0,0})$.  We need to show that we can assume $k_0(\x_1)$ and $k_0(\x_2)$ are contained in $K$.  By assumption, we have
\[0 = f_1(\y) = f_1(w_1z_2 \cdot k_0(\x_1) + z_2 w_1 \cdot k_0(\x_2)) = w_1z_2 \cdot f_1(k_0(\x_1)) + z_2w_1 \cdot f_1(k_0(\x))\]
The direct sum decomposition $\GHt(\GG_{1,1}) = w_1z_2 \GHt(\GG_{1,1}) \oplus z_2w_1 \GHt(\GG_{1,1})$ implies that 
\[f_1(w_1z_2 \cdot k_0(\x_1)) = 0 \qquad f_1(z_2 w_1 \cdot k_0(\x_2)) = 0\]

Since $f_1( w_1 z_2 \cdot k_0 (\x_1)) = 0$, we can apply the anticommutation relations to obtain
\begin{align}
0&= (w_T + z_T) \cdot f_1( w_1 z_2 \cdot k_0 (\x_1) ) \nonumber  \\
&= f_1\left( (1 + w_T + z_T) \cdot w_1 z_2 \cdot k_0 (\x_1) \right) \nonumber  \\
&= f_1( w_1z_2  \cdot (1 + w_T + z_T) \cdot k_0 (\x_1)  + f_1((w_1 + z_2) \cdot k_0 (\x_1)) \nonumber \\
&= f_1(w_1z_2 \cdot k_0( (w_T + z_T) \cdot \x_1) + (w_1 + z_2) \cdot f_1(k_0(\x_1))  \nonumber \\
&= (w_1 + z_2) \cdot f_1(k_0(\x_1))
\label{eq:double-vanish}
\end{align}
Again, the final equality follows from the previous line since $w_T + z_T$ is identically 0 on $\GHt(\GG_{0,0})$.  Since $w_1 + z_2$ is an invertible operator on $\GHt(\GG_{1,1})$, Equation \ref{eq:double-vanish} implies that $f_1(k_0(\x_1)) = 0$.  The same argument also proves that $f_1(k_0(\x_2)) = 0$.  This proves Equation \ref{eq:kernel-condition} and therefore the lemma.
\end{proof}

Combining the above two lemmata, we can prove Theorem \ref{thrm:HFK-mutation-invariance}.

\begin{proof}[Proof of Theorem \ref{thrm:HFK-mutation-invariance}]
Exactness of the skein triangle implies that
\begin{align*}
\GHt(\GG_{1,\infty}) &\cong \ker (f_1) [1] \oplus \Coker (f_1) \\ 
\GHt(\GG_{\infty,1}) &\cong \ker (k_1)[1] \oplus \Coker (k_1)
\end{align*}
There is a $\delta$--graded isomorphism between $\GHt(\GG_{1,0})$ and $\GHt(\GG_{0,1})$ since they are connected sums of the same two links.  Moreover, by Lemma \ref{lemma:double-rank} the $\delta$--graded ranks of $f_1$ and $k_1$ agree.  By the rank--nullity theorem, the $\delta$--graded ranks of $\ker (f_1)$ and $\ker (k_1)$ agree and similarly the $\delta$--graded ranks of $\Coker (f_1)$ and $\Coker (k_1)$ agree.  Thus, the $\delta$--graded ranks of $\GHt(\GG_{1,\infty})$ and $\GHt(\GG_{\infty,1})$ agree and there is a $\delta$--graded isomorphism of $\HFKh(\cL_{1,\infty})$ and $\HFKh(\cL_{\infty,1})$.
\end{proof}

\subsection{The Kinoshita-Terasaka family}
\label{sub:KT}

Kinoshita and Terasaka introduced a family $\{KT_{r,n}\}$ of knots with trivial Alexander polynomial \cite{KT}.  The knot $KT_{r,n}$ is obtained from the pretzel knot $P(-r,r+1,r,-r-1)$ by adding $n$ full twists.  See Figure \ref{fig:KT} for a diagram of $K_{3,2}$.  The knots are nontrivial for $n \neq 0$.  There is also a family $\{C_{r,n}\}$ of Conway mutants.  These are obtained instead from the pretzel knot $P(-r,r+1,-r-1,r)$ in a similar fashion.  There is a Conway sphere of $KT_{r,n}$ contain the $r$ and $-r-1$ twist regions and the knot $C_{r,n}$ is obtained by mutating the tangle inside this Conway sphere.  Let $T_r$ denote this tangle.  

For each $r,n \in \ZZ$ and $n \neq 0$, the knots $KT_{r,n}$ and $C_{r,n}$ are distinguished by their bigraded $\HFKh$ groups.

\begin{figure}[htpb!]
\centering
\includegraphics[width=.5\textwidth]{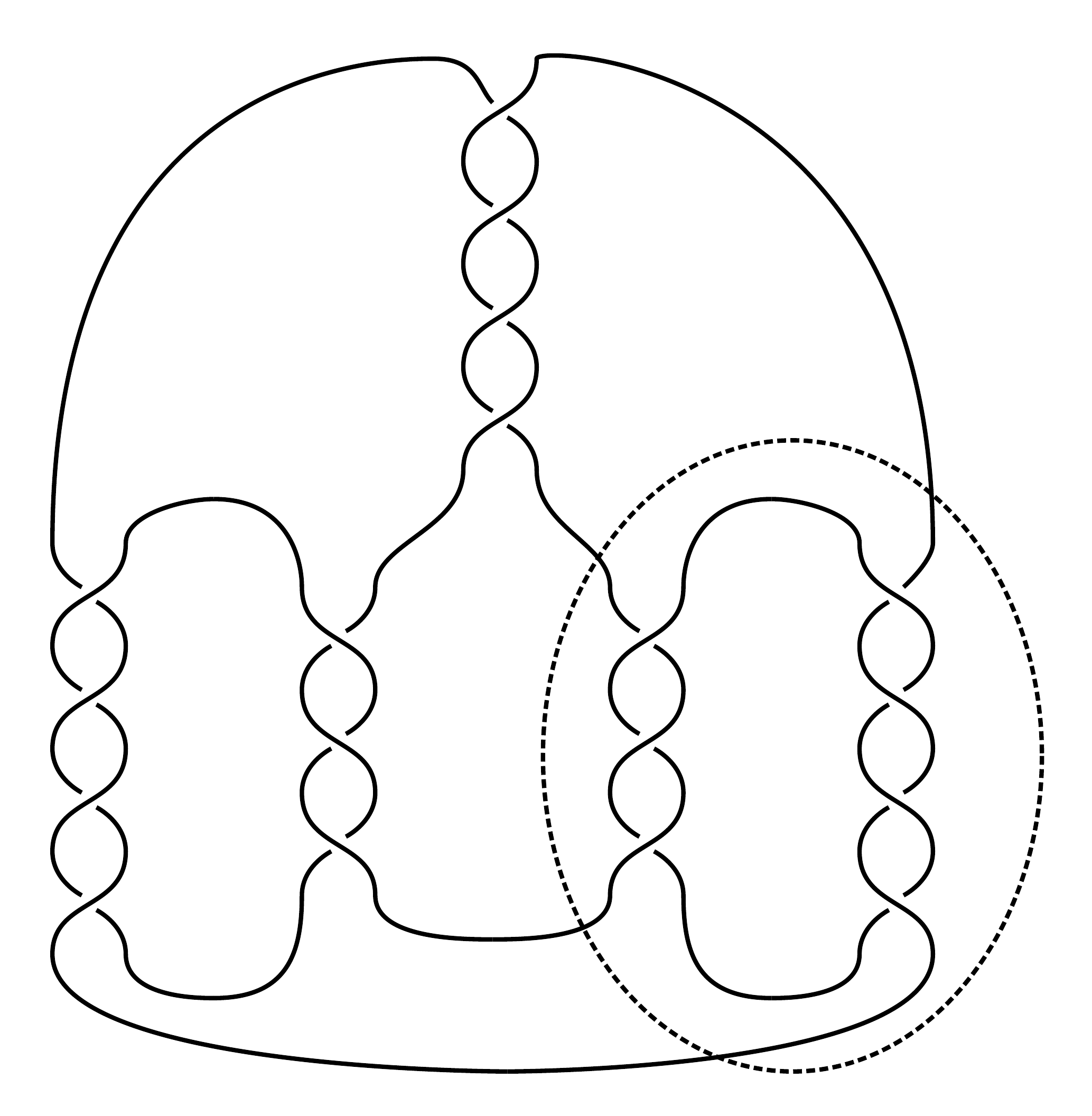}
\caption{The knot $KT_{3,2}$ with Conway sphere marked on the right.  The knot $C_{3,2}$ is obtained by mutating around the vertical axis in the page.}
\label{fig:KT}
\end{figure}

\begin{theorem}[\cite{OS-mutation}]
For $r,n \in \ZZ$ let $KT_{r,n}$ and $C_{r,n}$ denote the corresponding Kinoshita-Terasaka and Conway knots.
\begin{enumerate}
\item The bigraded knot Floer groups $\HFKh(KT_{r,n},s)$ vanish for $|s| > |r|$ and 
\[\HFKh(KT_{r,n},|r|) \cong \ZZ^{2n}\]
\item The bigraded knot Floer groups $\HFKh(C_{r,n},s)$ vanish for $|s| > 2|r|-1$ and 
\[\HFKh(C_{r,n},2|r|-1) \cong \ZZ^{2n}\]
\end{enumerate}
\end{theorem}

The  mutation on $T_r$ by rotating around the horiztonal axis in the figure is trivial.  Thus the two remaining mutations both give $C_{r,n}$.  This mutation is often chosen to be the positive mutation, which is rotation around the axis perpendicular to the diagram.  However, to apply Theorem \ref{thrm:HFK-mutation-invariance}, we choose the equivalent mutation that is rotation around the vertical axis in the page.

\begin{lemma}
\label{lemma:KT-Q-closure}
The set of rational closures $\cC_{\tau}(T_r)$ corresponding to mutation around the vertical axis contains the unknot.
\end{lemma}

\begin{proof}
Take the numerator closure of the tangle in Figure \ref{fig:KT}.  It is the knot $T(2,n)$ for $n = r + (-r-1) = -1$ and therefore it is the unknot.
\end{proof}

Combining Lemma \ref{lemma:KT-Q-closure} with Theorem \ref{thrm:HFK-mutation-invariance} proves Theorem \ref{thrm:KT-mutation} and so for all $r,n \in \ZZ$ there is an isomorphism 
\[\HFKh_{\delta}(KT_{r,n}) \cong \HFKh_{\delta}(C_{r,n})\]

\subsection{Low-crossing mutants}

Mutant cliques of 11-- and 12--crossing knots have been classified \cite{DeWit-Links,Stoimenow}.  Some cliques are composed of alternating knots, whose $\delta$--graded groups are determined by the determinant and signature.  These are mutation--invariant and thus the $\delta$--graded homology is invariant.  For each of the nonalternating cliques, the mutations can be achieved on one of a few tangles.

\begin{lemma}
\label{lemma:11-12}
Each nontrivial mutation of knots with crossing number $\leq 12$ can be obtained by a mutation on one of the 3 tangles in Figure \ref{fig:3-tangles}.
\end{lemma}

\begin{proof}
This can be checked by inspecting the minimal crossing diagrams for mutant cliques in \cite{DeWit-Links}.
\end{proof}

\begin{figure}[htpb!]
\centering
\includegraphics[width=.8\textwidth]{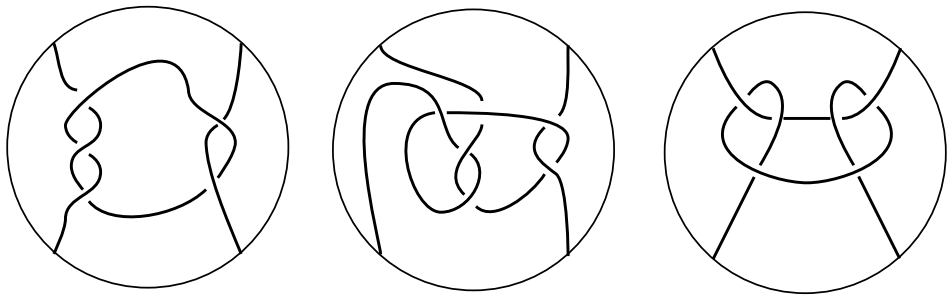}
\caption{Three tangles}
\label{fig:3-tangles}
\end{figure}

Note that each of the three tangles is the horizontal sum of two rational tangles.  Thus, rotating the sum around the horizontal axis is equal to rotating each rational tangle around the horizontal axis.  However, mutation on rational tangles does not change the isotopy class of a link.  Thus, mutation on the sum of rational tangles by rotating around the horizontal axis does not change the isotopy class either.  As a result, this implies that the remaining two mutations are identical up to isotopy.  Any mutation on the tangles of Figure \ref{fig:3-tangles} can be achieved by rotation around the vertical axis.

Taking the numerator closure, which corresponds to mutation around the vertical axis, of these tangles gives an unlink.  Kauffman and Lambropoulou give the following construction of unknots with complicated diagrams \cite[Theorem 5]{Kauffman-Lambro}.  Let $\left[\frac{p}{q}\right]$ and $\left[\frac{r}{s} \right]$ denote rational tangles determined by continued fraction expansions of $\frac{p}{q}$ and $\frac{r}{s}$, respectively.  Let $A = \left[ \frac{p}{q} \right] + \left[\frac{r}{s}\right]$ be their sum and let $L = N(A)$ be the numerator closure of $A$.  Set $c = ps + qr$ and $d = vs + ur$ where $|pu - qv| = 1$.  Then $L$ is isotopic to the numerator closure of $\left[\frac{c}{d} \right]$ and is therefore a 2-bridge link.  Note that $A$ is not a rational tangle, but its numerator closure is nonetheless a 2--bridge link.  Consequently, if $c = ps + qr = \pm 1$, then by Schubert's Theorem, $L$ is the unknot.  If $\frac{r}{s} = -\frac{p}{q}$, then $L$ is the unlink of 2 components.

\begin{lemma}
\label{lemma:U-closure}
Each of the tangles in Figure \ref{fig:3-tangles} has an unlinked rational closure.
\end{lemma}

\begin{proof}
This can be verified by inspecting a diagram for the numerator closures.  Moreover, the Kauffman-Lambropoulou result applies.  Tangle 1 is the sum of the rational tangles $[-\frac{1}{3}]$ and $[\frac{1}{2}]$; Tangle 2 is the sum of $[-\frac{3}{5}]$ and $[\frac{1}{2}]$; and Tangle 3 is the sum of $[\frac{2}{3}]$ and $[-\frac{2}{3}]$.  Thus, their numerator closures are the unknot, unknot and 2--component unlink, respectively.
\end{proof}

Theorem \ref{thrm:low-crossing} now follows by combining Lemmas \ref{lemma:11-12} and \ref{lemma:U-closure} with Theorem \ref{thrm:HFK-mutation-invariance}.

\section{Khovanov-Floer Theories}
\label{sec:KF}

\subsection{Khovanov-Floer theories}

Let $\Link$ denote the {\it link cobordism category}.  The objects of $\Link$ are oriented links in $S^3$ and the morphisms are isotopy classes of oriented link cobordims in $S^3 \times [0,1]$.  For a point $p \in S^3$, let $\Link_p$ denote the {\it based link cobordism category}.  The objects of $\Link_p$ are oriented links in $S^3$ containing $p$ and the morphisms are isotopy classes of oriented link cobordisms in $S^3 \times [0,1]$ containing the arc $p \times [0,1]$.

Let $\Diag$ denote the {\it diagrammatic link cobordism category}.  The objects of $\Diag$ are oriented link diagrams in $\RR^2$ and the morphisms are equivalence classes of movies.  A {\it movie} of oriented link diagrams is a family of link diagrams $D_t$ for $t \in [0,1] \setminus \{t_1,\dots,t_k\}$ such that (1) for $t \in (t_i,t_{i+1})$ the family $D_t$ is given by planar isotopy, and (2) for each $t_i$, the diagrams $D_{t_i-\epsilon}$ and $D_{t_i + \epsilon}$ are related by a Reidemeister mover or an elementary topological handle attachment.  For a point $p \in R^2$, let $\Diag_p$ denote the {\it based diagrammatic link cobordism category}, whose objects are oriented link diagrams containing the point $p$ and whose morphisms are movies of diagrams containing $p$.  See \cite{BHL-KF}.

Let $\FF = \ZZ/2\ZZ$ and let $\Vect_{\FF}$ denote the category of $\FF$--vector spaces and let $\Spect_{\FF}$ denote the category of spectral sequences with $\FF$--coefficients.  Recall that a spectral sequence is a sequence of chain complexes $ \{(E^i,d^i)\}_{i \geq i_0}$ satisfying $E^{i+1} \cong H_*(E^i,d^i)$ for all $i \geq i_0 + 1$.  A {\it morphism} of spectral sequences $F: \{(E^i,d^i)\}_{i \geq i_0} \rightarrow \{((E')^i,(d')^i)\}_{i \geq i_1}$ is a collection of chain maps
\[F_i: (E^i,d^i) \rightarrow \left((E')^i,(d')^i\right) \qquad \text{for }i \geq \text{ max}(i_0,i_1)\]
such that $F_{i+1} = (F_i)_*$.  For each $i \geq 0$, there is a forgetful functor $SV_i: \Spect_{\FF} \rightarrow \Vect_{\FF}$ that sends a spectral sequence to its $i^{\text{th}}$ page.

Khovanov homology determines a functor
\[ \Kh: \Diag \rightarrow \Vect_{\FF}\]
and reduced Khovanov homology determines a functor
\[\Khr: \Diag_p \rightarrow \Vect_{\FF}\]

Let $D$ be a link diagram.  A {\it $\Kh(D)$--complex} is a pair $ (C,q)$ consisting of
\begin{enumerate}
\item a $\ZZ$--filtered complex $C$, and
\item a graded vector space isomorphism $q: \Kh(D) \rightarrow E_2(C)$
\end{enumerate}

Let $D,D'$ be link diagrams and let $g: \Kh(D) \rightarrow \Kh(D')$ be a map of graded vector spaces that is homogeneous of degree $k$.  A chain map $f: (C,q) \rightarrow (C',q')$ of degree $k$ {\it agrees on $E^2$} with $g$ if the induced map
\[E^2(f): E^2(C) \rightarrow E^2(C')\]
satisfies $E^2(f) = q' \circ g \circ q^{-1}$ and so we have a commutative diagram
\[ \xymatrix{
E_2(C) \ar[r]^{E_2(f)} \ar[d]^{q^{-1}} & E_2(C) \\
\Kh(D) \ar[r]^{g} & \Kh(D') \ar[u]^{q'}
}\]

Two $\Kh(D)$--complexes $(C,q)$ and $(C',q')$ are {\it quasi--isomorphic} if there exists a degree 0 filtered chain map $f: C \rightarrow C'$ that agrees on $E^2$ with the identity map on $\Kh(D)$.  If $(C,q)$ and $(C',q')$ are quasi--isomorphic $\Kh(D)$-complexes, then there are canonical isomorphisms
\[I_r: (E^r(C),d^r) \rightarrow (E^r(C'), (d')^r) \qquad \text{for }r \geq 2\]
on all pages of the spectral sequence subsequent to $E^2$.  In addition, suppose $g: \Kh(D_1) \rightarrow \Kh(D_2)$ is a graded vector space map that is homogeneous of degree $k$.  Let $(C_1,q_1)$ and $(C'_1,q'_1)$ be quasi--isomorphic $\Kh(D_1)$--complexes and $(C_2,q_2)$ and $(C'_2,q'_2)$ be quasi--isomorphic $\Kh(D_2)$--complexes.  Suppose that $f: C_1 \rightarrow C_2$ and $f': C'_1 \rightarrow C'_2$ are filtered chain maps of degree $k$ that each agree on $E^2$ with $g$.  Then the induced maps $E^r(f)$ and $E^r(f')$ commute with the canonical isomorphisms for $r \geq 2$.  In particular, the map $f$ canonically determines a spectral sequence map that depends only on the quasi--isomorphism classes of $C_1$ and $C_2$.  

\begin{definition}[\cite{BHL-KF}]
\label{def:KFr}
A {\it Khovanov-Floer theory} $\cA$ is a rule that assigns to every link diagram $D$ a quasi--isomorphism class of $\Kh(D)$--complexes $\cA(D)$ such that
\begin{enumerate}
\item if $D$ and $D'$ are related by a planar isotopy, then there exists a morphism
\[\cA(D) \rightarrow \cA(D')\]
that agrees on $E^2$ with the induced map from $\Kh(D)$ to $\Kh(D')$,
\item if $D$ and $D'$ are related by a diagrammatic 1--handle attachment, then there exists a morphism
\[\cA(D) \rightarrow \cA(D')\]
that agrees on $E^2$ with the induced map from $\Kh(D)$ to $\Kh(D')$,
\item for any two link diagrams $D,D'$, there exists a morphism
\[\cA(D \cup D') \rightarrow \cA(D) \otimes \cA(D')\]
that agrees on $E^2$ with the standard isomorphism 
\[\Kh(D \cup D') \rightarrow \Kh(D) \otimes \Khr(D'), \]
\item for any diagram $D$ of the unlink $\cU_n$, $E^2(\cA(D)) = \dots = E^{\infty}(\cA(D))$.
\end{enumerate}
\end{definition}

A {\it reduced Khovanov-Floer theory} $\cA^r$ is defined similarly, except using $\Khr(D)$--complexes and Axiom (3) in Definition \ref{def:KFr} is replaced by a corresponding statement for connected sums instead of disjoint unions.

From the axioms in Definition \ref{def:KFr}, Baldwin, Hedden and Lobb prove that each page of the spectral sequence associated to Khovanov-Floer theory is a functorial link invariant.

\begin{theorem}[Baldwin-Hedden-Lobb \cite{BHL-KF}]
The spectral sequence associated to a Khovanov-Floer theory $\cA$ defines a functor
\[F_{\cA}: \Link \rightarrow \Spect_{\FF}\]
satisfying $\Kh = SV_2 \circ F_{\cA}$
\end{theorem}

If $\cA$ is a Khovanov-Floer theory, we denote $i^{\text{th}}$ link invariant associated to $\cA$ by $\cA_i\coloneqq SV_i \circ F_{\cA}$.

Often, a Khovanov-Floer theory also satisfies an unoriented skein exact sequence.  Let $L$ be a link with diagram $D$.  For a fixed crossing in $D$, we obtain two links $L_0$ and $L_1$ by taking the 0--resolution and 1--resolution of $L$ at the chosen crossing.  The three links $(L,L_1,L_0)$ are related by elementary 1--handle attachments.  There are three corresponding elementary cobordisms
\[f: L \rightarrow L_0 \qquad g: L_0 \rightarrow L_1 \qquad h: L_1 \rightarrow L\]
A functor $\cA: \Link \rightarrow \Vect_{\FF}$ satisfies an unoriented skein exact sequence if the triangle
\[\xymatrix{
\cA(L) \ar[rr]^{\cA(f)} & & \cA(L_0) \ar[dl]^{\cA(g)} \\
& \cA(L_1) \ar[ul]^{\cA(h)} & 
}\]
is exact for every triple $(L,L_0,L_1)$.

\subsection{Extended Khovanov-Floer theories}

Recall from Subsection \ref{sub:intro-KF} the notion of extended Khovanov-Floer theory.  An {\it extended Khovanov-Floer theory} is a pair $\cA,\cA^r$ consisting of an unreduced and reduced Khovanov-Floer theories, respectively, satisfying the following 3 extra axioms:
\begin{enumerate}
\item $\cA(L) = \cA^r(L \cup U,p)$ for a basepoint $p$ on the unknot component $U$,
\item $\cA$ and $\cA^r$ satisfy unoriented skein exact triangles, and
\item up to isomorphism, $\cA^r(L,p)$ is independent of the component containing $p$.
\end{enumerate}

Basepoint independence implies that every extended Khovanov-Floer theory satisfies a Kunneth--type principle for arbitrary connected sums.
\begin{lemma}
\label{lemma:KF-Kunneth}
Let $\cA,\cA^r$ be an extended Khovanov-Floer theory.  Then for any pair of oriented links the invariants satisfy 
\begin{align*}
\cA^r(L_1 \# L_2) & \cong \cA^r(L_1) \otimes \cA^r(L_2) \\
\cA(L_1 \# L_2) \otimes \cA(U) &\cong \cA(L_1) \otimes \cA(L_2) 
\end{align*}
for any choice of connected sum.
\end{lemma}
\begin{proof}
We will prove the lemma first for the reduced theory $\cA^r$.  The corresponding statement for $\cA$ follows since $\cA^r$ determines the unreduced theory.  Choose simultaneous diagrams $D_1,D_2$ for $(L_1,p_1)$ and $(L_2,p_2)$ so that there is an arc $a$ from $p_1$ to $p_2$ in the plane disjoint from the projections of $L_1$ and $L_2$.  Orient $L_1$ and $L_2$ so that the 1-handle attachement along $a$ is oriented.  Let $D$ be the corresponding diagram for $L_1 \# L_2$ with a single basepoint $p$ on the merged component.  Then axiom (3) for a reduced Khovanov-Floer theory states that there is a morphism
\[m: \cA^r(D,p) \rightarrow \cA^r(D_1,p_1) \otimes \cA^r(D_2,p_2)\]
that agrees on $E^2$ with the corresponding isomorphism for Khovanov homology.  Then \cite[Lemma 2.1]{BHL-KF} implies that this morphism is in fact an isomorphism.  Thus $\cA^r(L_1 \# L_2,p) = \cA^r(L_1,p) \otimes \cA^r(L_2,p_2)$.  A different choice of connected sum corresponds to different choices of $p_1,p_2$.  However, $\cA^r(L_1)$ and $\cA^r(L_2)$ are independent of the basepoint choices and so $\cA^r(L_1 \# L_2)$ is independent of the choice of connected sum.
\end{proof}

In order to establish mutation invariance, we make use of the skein exact triangle applied to the familiar triple
\[ L_1 \# -L_2 \qquad L_1 \cup L_2 \qquad  L_1 \# L_2\]
for a pair of oriented links $L_1,L_2$.  The following lemma, analogous to Lemma \ref{lemma:Kh-disjoint-split} for Khovanov homology over $\ZZ/2\ZZ$, follows by an identical argument.

\begin{lemma}
\label{lemma:KF-disjoint-split}
Let $\cA,\cA^r$ be an extended Khovanov-Floer theory.  Fix a pair of oriented links $L_1,L_2$ and an arc $a$ from $L_1$ to $L_2$.  Then the merge maps corresponding to an elementary 1--handle attachment along $a$
\begin{align*}
\mu^r_a: &\cA^r(L_1 \cup L_2) \rightarrow \cA^r(L_1 \# L_2) & 
\mu_a:&\cA(L_1 \cup L_2)  \rightarrow \cA(L_1 \# L_2)
\end{align*}
are surjective.
\end{lemma}

Using the topological results from Subsection \ref{sub:mutation}, we can now prove that every extended Khovanov-Floer theory is mutation--invariant.

\begin{proof}[Proof of Theorem \ref{thrm:KF-mutation-invariance}]

According to Lemma \ref{lemma:standard}, we can choose diagrams for $L,L'$ in standard form.   Take the 9 links in Figure \ref{fig:9-mutation} obtained by the various resolutions of $\cL_{\infty,\infty}$.  From the elementary pointed cobordisms of the skein exact triangle, we obtain a commutative diagram where each row and column is exact.

Exactness implies that 
\begin{align*}
\rk \cA(\cL_{1,\infty}) &= \rk \cA(\cL_{1,0}) + \rk \cA(\cL_{1,1})) - 2 \cdot \rk \Image (f_1) \\
\rk \cA(\cL_{\infty,1}) &= \rk \cA(\cL_{0,1}) + \rk  \cA(\cL_{1,1}))- 2 \cdot \rk \Image (k_1)
\end{align*}
However, by Lemma \ref{lemma:KF-Kunneth}, we have that
\[\cA(\cL_{1,0}) \cong \cA(\cL_{0,1})\]
and so they have the same rank.  Moreover, $k_0$ and $f_0$ are surjective by Lemma \ref{lemma:KF-disjoint-split} and commutativity implies that $f_1 \circ k_0 = k_1 \circ f_0$.  Thus
\[\Image (f_1) = \Image (f_1 \circ k_0) = \Image (k_1 \circ f_0) = \Image (k_1)\]
Consequently, $\rk \cA(\cL_{1,\infty}) = \rk \cA(\cL_{\infty,1})$ and thus $\cA(\cL_{1,\infty})$ and $\cA(\cL_{\infty,1})$ are isomorphic.
\end{proof}

Finally, the proofs of Lemma \ref{lemma:Hopf-QQ-surjective}, Theorem \ref{thrm:Kh-QQ-mutation-invariance} and Theorem \ref{thrm:KT-mutation-Khovanov} can be repeated {\it mutatis mutandis} by replacing Khovanov homology with singular instanton homology.  This proves Theorem \ref{thrm:SIH-mutation}.

\section{Discussion}

The geometric arguments in Section \ref{sec:HFK} imply that Conjecture \ref{conj:mutation-invariance} will follow if there is a homotopy equivalence
\[\Cone( f_1: \CFKt(\cL_{1,0}) \rightarrow \CFKt(\cL_{1,1})) \sim \Cone( k_1: \CFKt(\cL_{0,1}) \rightarrow \CFKt(\cL_{1,1}))\]
Using the Clifford module structure, we approximate this by equating the ranks of the induced maps on homology when the mutated tangle is sufficiently simple.  Extending this result to arbitrary tangles would completely prove the conjecture.

However, the extra hypothesis of Theorem \ref{thrm:HFK-mutation-invariance} may be geometrically relevant.  First, as Theorem \ref{thrm:low-crossing} indicates, the condition on tangle closures explains some but not all of the computational evidence for Conjecture \ref{conj:mutation-invariance}.  Most low--crossing tangles can be closed off to the unlink and so most low--crossing mutant pairs should satisfy the hypotheses of Theorem \ref{thrm:HFK-mutation-invariance}.  Thus, Theorem \ref{thrm:low-crossing} may be viewed as minor inclupatory evidence against the conjecture.

Secondly, Zibrowius has shown a stronger result that positive mutations on the $(3,-2)$--pretzel tangle, the first tangle in Figure \ref{fig:3-tangles}, preserves bigraded $\HFKh$ \cite{Zibrowius}.  Specifically, if the tangle appears in a link with both strands oriented upwards, then mutating around the $y$-axis preserves not just the $\delta$--graded invariant (which is guaranteed by Theorem \ref{thrm:HFK-mutation-invariance}) but the full bigraded invariant.  However, the $(3,-2)$--pretzel tangle is abstractly diffeomorphic to the $(3,2)$-pretzel tangle.  If this latter tangle appears in a link with both strands oriented upwards, as it does in the Kinoshita-Terasaka knot, then mutating around the $y$-axis {\it does not} preserve bigraded $\HFKh$.  One speculative explanation is that the numerator closure of the $(3,-2)$--pretzel is the unknot while the numerator closure of the $(3,2)$--pretzel is the right--handed cinquefoil $T(2,5)$.  Based on this observation, we conjecture a stronger version of Theorem \ref{thrm:HFK-mutation-invariance}.

\begin{conjecture}
Let $L,L',T$ be mutant links and a tangle satisfying the hypotheses of Theorem \ref{thrm:HFK-mutation-invariance}.  If the mutation is positive, then there is a bigraded isomorphism
\[\HFKh(L) \cong \HFKh(L')\]
\end{conjecture}

In a different direction, the key fact necessary to prove Theorem \ref{thrm:HFK-mutation-invariance} is that all basepoint maps vanish on $\HFKh(\cU_k)$.  Theorem \ref{thrm:HFK-mutation-invariance} can be extended to tangles $T$ where the basepoint maps vanish on $\HFKh(C(T))$ for some rational closure $C(T)$.  The basepoint maps will vanish on $\HFKh(C(T))$ if the knot $C(T)$ has no length-1 differentials in its $\CFK^{\infty}$ complex.   However, we know of no knots beside the unknot which have this property.

\begin{question}
Does there exist a nontrivial knot $K$ such that $\CFK^{\infty}(K)$ has no length-1 differentials?
\end{question}

Combining \cite[Theorem 1.2]{OS-L-space} with \cite[Corollary 9]{HW-GB} proves that no nontrivial L-space knot has this property.  Also, as pointed out to the author by Jen Hom, there exist knots --- for example $T_{4,5} \# -T_{2,3;2,5}$ --- whose $\CFK^{\infty}$ complex has a {\it direct summand} with no length-1 differentials \cite{Hom}.

\bibliographystyle{alpha}
\bibliography{References}

\begin{thebibliography}{MOST07}

\bibitem[AP04]{Asaeda-Przytycki}
Marta~M. Asaeda and J{\'o}zef~H. Przytycki.
\newblock Khovanov homology: torsion and thickness.
\newblock In {\em Advances in topological quantum field theory}, volume 179 of
  {\em NATO Sci. Ser. II Math. Phys. Chem.}, pages 135--166. Kluwer Acad.
  Publ., Dordrecht, 2004.

\bibitem[Bal11]{Baldwin-2cover}
John~A. Baldwin.
\newblock On the spectral sequence from {K}hovanov homology to {H}eegaard
  {F}loer homology.
\newblock {\em Int. Math. Res. Not. IMRN}, (15):3426--3470, 2011.

\bibitem[BG12]{BG-computations}
John~A. Baldwin and William~D. Gillam.
\newblock Computations of {H}eegaard-{F}loer knot homology.
\newblock {\em J. Knot Theory Ramifications}, 21(8):1250075, 65, 2012.

\bibitem[BHL]{BHL-KF}
John~A. Baldwin, Matthew Hedden, and Andrew Lobb.
\newblock On the functoriality of {K}hovanov-{F}loer theories.

\bibitem[BL12]{BL-spanning}
John~A. Baldwin and Adam~Simon Levine.
\newblock A combinatorial spanning tree model for knot {F}loer homology.
\newblock {\em Adv. Math.}, 231(3-4):1886--1939, 2012.

\bibitem[Blo10]{Bloom-Odd-Khovanov}
Jonathan~M. Bloom.
\newblock Odd {K}hovanov homology is mutation invariant.
\newblock {\em Math. Res. Lett.}, 17(1):1--10, 2010.

\bibitem[BLS]{BLS}
John~A. Baldwin, Adam~Simon Levine, and Sucharit Sarkar.
\newblock Khovanov homology and knot {F}loer homology for pointed links.

\bibitem[BN05]{BN-Kh}
Dror Bar-Natan.
\newblock Khovanov's homology for tangles and cobordisms.
\newblock {\em Geom. Topol.}, 9:1443--1499, 2005.

\bibitem[BVVV13]{BVV}
John~A. Baldwin, David~Shea Vela-Vick, and Vera V{\'e}rtesi.
\newblock On the equivalence of {L}egendrian and transverse invariants in knot
  {F}loer homology.
\newblock {\em Geom. Topol.}, 17(2):925--974, 2013.

\bibitem[DWL07]{DeWit-Links}
David De~Wit and Jon Links.
\newblock Where the {L}inks-{G}ould invariant first fails to distinguish
  nonmutant prime knots.
\newblock {\em J. Knot Theory Ramifications}, 16(8):1021--1041, 2007.

\bibitem[Gab86]{Gabai-genera}
David Gabai.
\newblock Genera of the arborescent links.
\newblock {\em Mem. Amer. Math. Soc.}, 59(339):i--viii and 1--98, 1986.

\bibitem[Gre12]{Greene-alternating}
Joshua~Evan Greene.
\newblock Conway mutation and alternating links.
\newblock In {\em Proceedings of the {G}\"okova {G}eometry-{T}opology
  {C}onference 2011}, pages 31--41. Int. Press, Somerville, MA, 2012.

\bibitem[HN13]{Hedden-Ni}
Matthew Hedden and Yi~Ni.
\newblock Khovanov module and the detection of unlinks.
\newblock {\em Geom. Topol.}, 17(5):3027--3076, 2013.

\bibitem[Hom16]{Hom}
Jennifer Hom.
\newblock A note on the concordance invariants epsilon and upsilon.
\newblock {\em Proc. Amer. Math. Soc.}, 144(2):897--902, 2016.

\bibitem[HW]{HW-GB}
Matthew Hedden and Liam Watson.
\newblock On the geography and botany problem of knot floer homology.

\bibitem[JM]{Juhazs-Marengon}
Andr{\'a}s Juh{\'a}sz and Marco Marengon.
\newblock Computing cobordism maps in link {F}loer homology and the reduced
  {K}hovanov {TQFT}.

\bibitem[Kho00]{Khovanov-Jones}
Mikhail Khovanov.
\newblock A categorification of the {J}ones polynomial.
\newblock {\em Duke Math. J.}, 101(3):359--426, 2000.

\bibitem[KL12]{Kauffman-Lambro}
Louis~H. Kauffman and Sofia Lambropoulou.
\newblock Hard unknots and collapsing tangles.
\newblock In {\em Introductory lectures on knot theory}, volume~46 of {\em Ser.
  Knots Everything}, pages 187--247. World Sci. Publ., Hackensack, NJ, 2012.

\bibitem[KM11a]{KM-Kh}
P.~B. Kronheimer and T.~S. Mrowka.
\newblock Khovanov homology is an unknot-detector.
\newblock {\em Publ. Math. Inst. Hautes \'Etudes Sci.}, (113):97--208, 2011.

\bibitem[KM11b]{KM-SIH}
P.~B. Kronheimer and T.~S. Mrowka.
\newblock Knot homology groups from instantons.
\newblock {\em J. Topol.}, 4(4):835--918, 2011.

\bibitem[KT57]{KT}
Shin'ichi Kinoshita and Hidetaka Terasaka.
\newblock On unions of knots.
\newblock {\em Osaka Math. J.}, 9:131--153, 1957.

\bibitem[{Lam}16]{PLC}
P.~{Lambert-Cole}.
\newblock {Twisting, mutation and knot Floer homology}.
\newblock {\em ArXiv e-prints}, August 2016.

\bibitem[Man07]{Manolescu-skein}
Ciprian Manolescu.
\newblock An unoriented skein exact triangle for knot {F}loer homology.
\newblock {\em Math. Res. Lett.}, 14(5):839--852, 2007.

\bibitem[MO08]{MO-QA}
Ciprian Manolescu and Peter Ozsv{\'a}th.
\newblock On the {K}hovanov and knot {F}loer homologies of quasi-alternating
  links.
\newblock In {\em Proceedings of {G}\"okova {G}eometry-{T}opology {C}onference
  2007}, pages 60--81. G\"okova Geometry/Topology Conference (GGT), G\"okova,
  2008.

\bibitem[MOS09]{MOS}
Ciprian Manolescu, Peter Ozsv{\'a}th, and Sucharit Sarkar.
\newblock A combinatorial description of knot {F}loer homology.
\newblock {\em Ann. of Math. (2)}, 169(2):633--660, 2009.

\bibitem[MOST07]{MOST}
Ciprian Manolescu, Peter Ozsv{\'a}th, Zolt{\'a}n Szab{\'o}, and Dylan Thurston.
\newblock On combinatorial link {F}loer homology.
\newblock {\em Geom. Topol.}, 11:2339--2412, 2007.

\bibitem[MS15]{Moore-Starkston}
Allison~H. Moore and Laura Starkston.
\newblock Genus-two mutant knots with the same dimension in knot {F}loer and
  {K}hovanov homologies.
\newblock {\em Algebr. Geom. Topol.}, 15(1):43--63, 2015.

\bibitem[Ni14]{Ni}
Yi~Ni.
\newblock Homological actions on sutured {F}loer homology.
\newblock {\em Math. Res. Lett.}, 21(5):1177--1197, 2014.

\bibitem[OS]{OS-skein}
Peter~S. Ozsv{\'a}th and Zolt{\'a}n Szab{\'o}.
\newblock On the skein exact sequence for knot {F}loer homology.

\bibitem[OS03]{OS-alternating}
Peter Ozsv{\'a}th and Zolt{\'a}n Szab{\'o}.
\newblock Heegaard {F}loer homology and alternating knots.
\newblock {\em Geom. Topol.}, 7:225--254 (electronic), 2003.

\bibitem[OS04a]{OS-HFK}
Peter Ozsv{\'a}th and Zolt{\'a}n Szab{\'o}.
\newblock Holomorphic disks and knot invariants.
\newblock {\em Adv. Math.}, 186(1):58--116, 2004.

\bibitem[OS04b]{OS-HF}
Peter Ozsv{\'a}th and Zolt{\'a}n Szab{\'o}.
\newblock Holomorphic disks and topological invariants for closed
  three-manifolds.
\newblock {\em Ann. of Math. (2)}, 159(3):1027--1158, 2004.

\bibitem[OS04c]{OS-mutation}
Peter Ozsv{\'a}th and Zolt{\'a}n Szab{\'o}.
\newblock Knot {F}loer homology, genus bounds, and mutation.
\newblock {\em Topology Appl.}, 141(1-3):59--85, 2004.

\bibitem[OS05a]{OS-L-space}
Peter Ozsv{\'a}th and Zolt{\'a}n Szab{\'o}.
\newblock On knot {F}loer homology and lens space surgeries.
\newblock {\em Topology}, 44(6):1281--1300, 2005.

\bibitem[OS05b]{OS-2cover}
Peter Ozsv{\'a}th and Zolt{\'a}n Szab{\'o}.
\newblock On the {H}eegaard {F}loer homology of branched double-covers.
\newblock {\em Adv. Math.}, 194(1):1--33, 2005.

\bibitem[OS08]{OS-HFL}
Peter Ozsv{\'a}th and Zolt{\'a}n Szab{\'o}.
\newblock Holomorphic disks, link invariants and the multi-variable {A}lexander
  polynomial.
\newblock {\em Algebr. Geom. Topol.}, 8(2):615--692, 2008.

\bibitem[OSS15]{OSS-book}
Peter~S. Ozsv{\'a}th, Andr{\'a}s~I. Stipsicz, and Zolt{\'a}n Szab{\'o}.
\newblock {\em Grid homology for knots and links}, volume 208 of {\em
  Mathematical Surveys and Monographs}.
\newblock American Mathematical Society, Providence, RI, 2015.

\bibitem[Ras10]{Rasmussen-S}
Jacob Rasmussen.
\newblock Khovanov homology and the slice genus.
\newblock {\em Invent. Math.}, 182(2):419--447, 2010.

\bibitem[Sar11]{Sarkar-tau}
Sucharit Sarkar.
\newblock Grid diagrams and the {O}zsv\'ath-{S}zab\'o tau-invariant.
\newblock {\em Math. Res. Lett.}, 18(6):1239--1257, 2011.

\bibitem[Sar15]{Sarkar-basepoints}
Sucharit Sarkar.
\newblock Moving basepoints and the induced automorphisms of link {F}loer
  homology.
\newblock {\em Algebr. Geom. Topol.}, 15(5):2479--2515, 2015.

\bibitem[{See}11]{Seed-comp}
C.~{Seed}.
\newblock {Computations of Szab$\backslash$'o's Geometric Spectral Sequence in
  Khovanov Homology}.
\newblock {\em ArXiv e-prints}, October 2011.

\bibitem[Shu14]{Shumakovitch}
Alexander~N. Shumakovitch.
\newblock Torsion of {K}hovanov homology.
\newblock {\em Fund. Math.}, 225(1):343--364, 2014.

\bibitem[Sto10]{Stoimenow}
A.~Stoimenow.
\newblock Tabulating and distinguishing mutants.
\newblock {\em Internat. J. Algebra Comput.}, 20(4):525--559, 2010.

\bibitem[Sza15]{Szabo-GSS}
Zolt{\'a}n Szab{\'o}.
\newblock A geometric spectral sequence in {K}hovanov homology.
\newblock {\em J. Topol.}, 8(4):1017--1044, 2015.

\bibitem[Vir76]{Viro}
O.~Ja. Viro.
\newblock Nonprojecting isotopies and knots with homeomorphic coverings.
\newblock {\em Zap. Nau\v cn. Sem. Leningrad. Otdel. Mat. Inst. Steklov.
  (LOMI)}, 66:133--147, 207--208, 1976.
\newblock Studies in topology, II.

\bibitem[Wat07]{Watson}
Liam Watson.
\newblock Knots with identical {K}hovanov homology.
\newblock {\em Algebr. Geom. Topol.}, 7:1389--1407, 2007.

\bibitem[Weh]{Wehrli-KH-examples}
Stephan~M. Wehrli.
\newblock Khovanov homology and {C}onway mutation.

\bibitem[Weh10]{Wehrli-Kh-mutation}
Stephan~M. Wehrli.
\newblock Mutation invariance of {K}hovanov homology over {$\Bbb F_2$}.
\newblock {\em Quantum Topol.}, 1(2):111--128, 2010.

\bibitem[Won]{Wong-skein}
C.-M.~Mike Wong.
\newblock Grid diagrams and {M}anolescu's unoriented skein exact triangle for
  knot {F}loer homology.

\bibitem[{Zem}16]{Zemke-basepoints}
I.~{Zemke}.
\newblock {Quasi-stabilization and basepoint moving maps in link Floer
  homology}.
\newblock {\em ArXiv e-prints}, April 2016.

\bibitem[{Zib}16]{Zibrowius}
C.~{Zibrowius}.
\newblock {On a Heegaard Floer theory for tangles}.
\newblock {\em ArXiv e-prints}, October 2016.

\end{thebibliography}


\end{document}